\documentclass[english,letter paper,12pt,leqno]{article}

\usepackage{array}
\usepackage{stmaryrd}
\usepackage{amsmath, amscd, amssymb, mathrsfs, accents, amsfonts,amsthm,cmll}
\usepackage[all]{xy}
\usepackage{dsfont}
\usepackage{bussproofs}
\usepackage{tikz}
\usepackage{epigraph}
\usepackage{tikz-cd}
\usepackage{bbm}
\usepackage{csquotes}

\usepackage{array}
\usepackage{units}
\usepackage{graphicx}
\usepackage{tikz-cd}
\usepackage{nicefrac}
\usepackage{hyperref}
\usepackage{bbm}
\usepackage{color}
\usepackage{tensor}
\usepackage{tipa}
\usepackage{bussproofs}
\usepackage{ stmaryrd }
\usepackage{ textcomp }
\usepackage{leftidx}
\usepackage{afterpage}
\usepackage{varwidth}
\usepackage{tasks}
\usepackage{makecell}
\usepackage{MnSymbol}
\usepackage{quiver}
\usepackage{adjustbox}
\usepackage{multirow}
\usepackage{booktabs}
\usepackage{xparse}
\usepackage{calc}

\definecolor{Myblue}{rgb}{0,0,0.6}
\newcommand{\cut}{\operatorname{cut}}
\newcommand{\ax}{\operatorname{ax}}

\usepackage{physics}

\SelectTips{cm}{}

\newtheorem{theorem}{Theorem}[section]
\newtheorem{thm}[theorem]{Theorem}
\newtheorem{proposition}[theorem]{Proposition}
\newtheorem{lemma}[theorem]{Lemma}
\newtheorem{corollary}[theorem]{Corollary}

\theoremstyle{definition}
\newtheorem{defn}[theorem]{Definition}

\newcommand{\tagarray}{\mbox{}\refstepcounter{equation}$(\theequation)$}

\newtheoremstyle{example}{\topsep}{\topsep}
	{}
	{}
	{\bfseries}
	{.}
	{2pt}
	{\thmname{#1}\thmnumber{ #2}\thmnote{ #3}}
	
	\theoremstyle{example}
	
	\newtheorem{example}[theorem]{Example}
	\newtheorem{remark}[theorem]{Remark}

\numberwithin{equation}{section}

\newcommand{\call}[1]{\mathcal{#1}}

\newcommand{\comment}[1]{}
\newcommand{\den}[1]{\llbracket #1 \rrbracket}

\def\be{\begin{equation}}
\def\ee{\end{equation}}

\begin{document}

% Bussproof things
\def\ScoreOverhang{1pt}

% a hack for vdots to decrease the vertical space above it; see https://tex.stackexchange.com/questions/169679/vdots-are-taller-than-the-rest-of-text
\makeatletter
\DeclareRobustCommand{\rvdots}{%
  \vbox{
    \baselineskip4\p@\lineskiplimit\z@
    \kern-\p@
    \hbox{}\hbox{.}\hbox{.}\hbox{.}
  }}
\makeatother

% Commands
\newcommand{\proofvdots}[1]{\overset{\displaystyle #1}{\rvdots}}
\newcommand{\ud}{\mathrm{d}}
\newcommand{\Ress}[1]{\res_{#1}\!}
\newcommand{\cat}[1]{\mathcal{#1}}
\newcommand{\lto}{\longrightarrow}
\newcommand{\xlto}[1]{\stackrel{#1}\lto}
\newcommand{\md}[1]{\mathscr{#1}}
\def\l{\,|\,}
\newcommand{\bb}[1]{\mathbb{#1}}
\newcommand{\scr}[1]{\mathscr{#1}}
\def\LG{\mathcal{LG}}
\def\hilb{\mathcal{W}}

\title{Linear Logic and Quantum Error Correcting Codes}
\author{Daniel Murfet, William Troiani}

\maketitle

\begin{abstract}
We develop a point of view on reduction of multiplicative proof nets based on quantum error-correcting codes. To each proof net we associate a code, in such a way that cut-elimination corresponds to error correction.
\end{abstract}

\tableofcontents

\setlength{\epigraphwidth}{0.8\textwidth}
\epigraph{But as artificers do not work with perfect accuracy, it comes to pass that mechanics is so distinguished from geometry that what is perfectly accurate is called geometrical; what is less so, is called mechanical. However, the errors are not in the art, but in the artificers.}{I.~Newton, \textsl{Principia}}

\section{Introduction}

In the tradition of proof theory, the structure of proofs lies jointly in the deduction rules and reduction rules \cite{negri}. The former say which occurrences of atomic propositions are the same while the latter say which \emph{proofs} are the same. For example, the Axiom and Cut links in linear logic
\begin{center}
	\begin{tabular}{>{\centering}m{6cm} >{\centering}m{6cm} >{\centering}m{0.5cm}}
		$
		\begin{tikzcd}[column sep = small, row sep = small]
			& \ax\arrow[dl,bend right, dash]\arrow[dr,bend left, dash]\\
			\neg A & & A
		\end{tikzcd}$
		&
		$
			\begin{tikzcd}[column sep = small, row sep = small]
				\neg A\arrow[dr,bend right] & & A\arrow[dl, bend left]\\
				& \cut
			\end{tikzcd}$
		&
		\tagarray{\label{eq:intro_deduction_a}}
	\end{tabular}
\end{center}
are the analogues of deduction rules in that system and they say ``$A$ is $A$, and conversely'' \cite[\S 3.2.1]{girard_blind}. This statement can be unpacked in terms of equalities between (positive and negative) occurrences of atomic propositions in the occurrences of the formula $A$ \cite{algpnt}. On the other hand, the cut-elimination transformation
\begin{center}
	\begin{tabular}{ c c c }
		$
		\begin{tikzcd}[column sep = small, row sep = small]
			\vdots\arrow[d,dash] &&& \ax\arrow[dl, bend right, dash]\arrow[dr,dash, bend left]\\
			A\arrow[dr, bend right] && \neg A\arrow[dl,bend left] && A\arrow[d]\\
			& \cut &&& \vdots
		\end{tikzcd}
		$
		&\qquad $\Longrightarrow$ \qquad
		$
		\begin{tikzcd}[column sep = small, row sep = small]
			\vdots\arrow[d,dash]\\
			A\arrow[d]\\
			\vdots
		\end{tikzcd}
		$
		&
		\qquad \tagarray{\label{rule:ax_reduct}}
	\end{tabular}
\end{center}
is an example of a reduction rule. It says that two proof nets which differ only in this way have ``the same'' logical content. From one point of view the purpose of denotational semantics of linear logic is to represent the patterns of equality arising from deduction rules in other kinds of mathematical objects. The field has produced many interesting examples, but the patterns of equality \emph{between proofs} coded in the reduction rules have proven harder to represent in a natural way.

If proofs $\pi, \psi$ related by a reduction are represented by mathematical objects $\den{\pi}, \den{\psi}$ should these objects be equal, or perhaps just isomorphic? In light of the Curry-Howard correspondence, which links the reduction process to computation \cite{howard,gmz} it seems desirable to find mathematical models in which this question has a natural and convincing answer. This desire was expressed as a coherent program in Girard's work on \emph{geometry of interaction} \cite{girard_goi}, which outlines a kind of mathematical model in which proofs are modelled by finite dynamical systems and computation is modelled by interaction. In this paper we give a simple realisation of these ideas for multiplicative proof nets, using finite-dimensional quantum systems called quantum error-correcting codes.
\\

A \emph{finite quantum system} is a pair $(\mathscr{H}, H)$ consisting of a finite-dimensional Hilbert space $\mathscr{H}$ and a self-adjoint operator $H$, the Hamiltonian, which is the generator of dynamics in the sense that the unitary operator $U(t) = \exp(-iht)$ gives the time evolution of states. In modern physics it is common to engineer novel quantum systems with desirable properties by carefully designing the Hamiltonian and then expending work to lower the system into a ground state whose properties derive from the information in $H$.

In a famous example due to Kitaev \cite{kitaev} the information in $H$ is a lattice representation of a torus, and the ground state of the quantum system has topologically protected order derived from the (co)homology of the torus; this is the famous \emph{toric code}. This part of physics provides us with an interesting example of a situation in which some information (the structure in $H$) is implicit and in order to manifest the order in the real world, that is, to make that implicit information explicit, you must perform work. In the case of the toric code, the work you perform is called \emph{quantum error-correction} \cite{nc}.
\\

In Section \ref{sec:Hilbert_Space_model} we assign to each proof $\pi$ in multiplicative linear logic a Hilbert space $\call{H}_\pi$, quantum error-correcting code $S_{\pi}$ and Hamiltonian $H_\pi = - \sum_{X \in S_{\pi}} X$. We think of
\[
	\den{\pi} = \big( \call{H}_\pi, S_\pi \big)
\]
as either the quantum error-correcting code or the associated finite quantum system with Hamiltonian $H_\pi$. An observer interacting with this system at high energy will be unable to extract the structure of the proof $\pi$ by analysing the statistics of their measurements. However, if we expend work to reduce the energy of the system far enough that the state $\ket{\Psi}$ lies in the (degenerate) space of ground states of $H_\pi$ (or what is the same, the code space of $S_\pi$) then the proof $\pi$ may be learned from measurements; see Remark \ref{remark:entangled_measure}. An abstract representation of the expenditure of work in this cooling process is our model of the process of computation.

For multiplicative proof nets the reduction rules are cut-elimination transformations. Given a proof $\pi_1 \l \pi_2$ obtained by cutting together a conclusion $A$ of $\pi_1$ with a conclusion $\neg A$ of $\pi_2$ a cut-elimination transformation produces a proof $\psi$ with a ``simpler'' arrangement of cuts. The question can then be be restated: can we find natural mathematical models $\den{-}$ of proofs, in which there is an interesting relationship among $\den{ \pi_1 \l \pi_2 }, \den{ \pi_1 }, \den{ \pi_2 }$ and $\den{\psi}$? In the situation described, we have systems
\begin{equation}
	(\mathscr{H}_1, H_1)\,, \quad (\mathscr{H}_2, H_2)\,,\quad (\mathscr{H}_{\psi}, H_\psi)\,.
\end{equation}
Assigned to $\pi = \pi_1 \l \pi_2$ is a system whose Hilbert space and Hamiltonian are
\begin{equation}
	\big( \mathscr{H}_1 \otimes \mathscr{H}_{cut} \otimes \mathscr{H}_2, H_1 + H_2 + H_{cut}\big)
\end{equation}
where $\mathscr{H}_{cut}$ is a system of qubits associated to the Cut link and $H_{cut}$ couples the degrees of freedom in the two systems to these qubits (and thus indirectly to each other). That is, $H_{cut}$ is the (only) source of interactions between the systems $\den{\pi_1}, \den{\pi_2}$. When two systems like this are brought together and made to interact the composite system will not usually be in a ground state. The states we can easily prepare are product states
\begin{equation}\label{eq:composite_state_0}
	\ket{\Psi_1} \otimes \ket{0} \otimes \ket{\Psi_2} \in \mathscr{H}_1 \otimes \mathscr{H}_{cut} \otimes \mathscr{H}_2\,.
\end{equation}
Work is required to cool the system from such a product state to a ground state of the interacting Hamiltonian $H_1 + H_2 + H_{cut}$. Even if $\ket{\Psi_1}$ and $\ket{\Psi_2}$ are ground states of their respective Hamiltonians (that is, they ``represent'' $\pi_1, \pi_2$) the product state \eqref{eq:composite_state_0} does not represent $\pi_1 \l \pi_2$ in the sense explained above. This motivates, from the physical point of view, the study of processes for cooling a state like \eqref{eq:composite_state_0}.

In Section \ref{section:reductions} we associate to each reduction $\gamma$ of the proof net $\pi$, for example $\pi_1 \l \pi_2 \rightsquigarrow \psi$, a cooling procedure $P_\gamma$. To each such procedure is associated the set of states
\[
V_\gamma \subseteq \mathscr{H}_1 \otimes \mathscr{H}_{cut} \otimes \mathscr{H}_2
\]
which are left invariant $P_\gamma\ket{\alpha} = \ket{\alpha}$. This subspace is closed under the interacting Hamiltonian, and the mathematical relationship between $\den{\pi} = \den{\pi_1 \l \pi_2}$ and $\den{\psi}$ which we put forward as representing the reduction $\gamma$ is an isomorphism of systems
\begin{equation}\label{eq:iso_of_systems}
(\mathscr{H}_{\psi}, H_\psi) \cong \big(V_\gamma, (H_1 + H_2 + H_{cut})|_{V_\gamma} \big)\,.
\end{equation}
More precisely, we give an explicit linear map
\[
T_\gamma: \mathscr{H}_\psi \lto \mathscr{H}_1 \otimes \mathscr{H}_{cut} \otimes \mathscr{H}_2
\]
which is an isometry onto the subspace $V_\gamma$, and which satisfies
\[
\big( H_1 + H_2 + H_{cut} \big) \circ T_\gamma = T_\gamma \circ H_{\psi}\,.
\]
It is interesting to note that this ``semantics'' of $\gamma$ has two parts:
\begin{itemize}
\item[(i)] Cooling the system $\den{\pi_1 \l \pi_2}$. This is a physical interaction, the details of which are not modelled here. Mathematically it can be represented (at a first brush) by $P_\gamma$.
\item[(ii)] Recognising $\den{\psi}$ as being isomorphic to the cooled system, by encoding the states and transitions of $\den{\psi}$ in $\den{\pi}$. This can be thought of as more information theoretic (and subjective) since it relates to our \emph{model} of the system and allows us to replace our complex model by a simpler one. Mathematically it is represented by $T_\gamma$.
\end{itemize}
A sequence of reductions leading from a proof $\pi$ to a normal form $\widetilde{\pi}$
\[
\xymatrix@C+2pc{
\pi = \pi_0 \ar@{~>}[r]^{\gamma_0} & \pi_1 \ar@{~>}[r]^{\gamma_1} & \quad \cdots \quad \ar@{~>}[r]^{\gamma_{n-1}} & \pi_n = \widetilde{\pi}
}
\]
corresponds to a diagram of nested Hilbert spaces
\begin{center}
\includegraphics[width=0.5\textwidth]{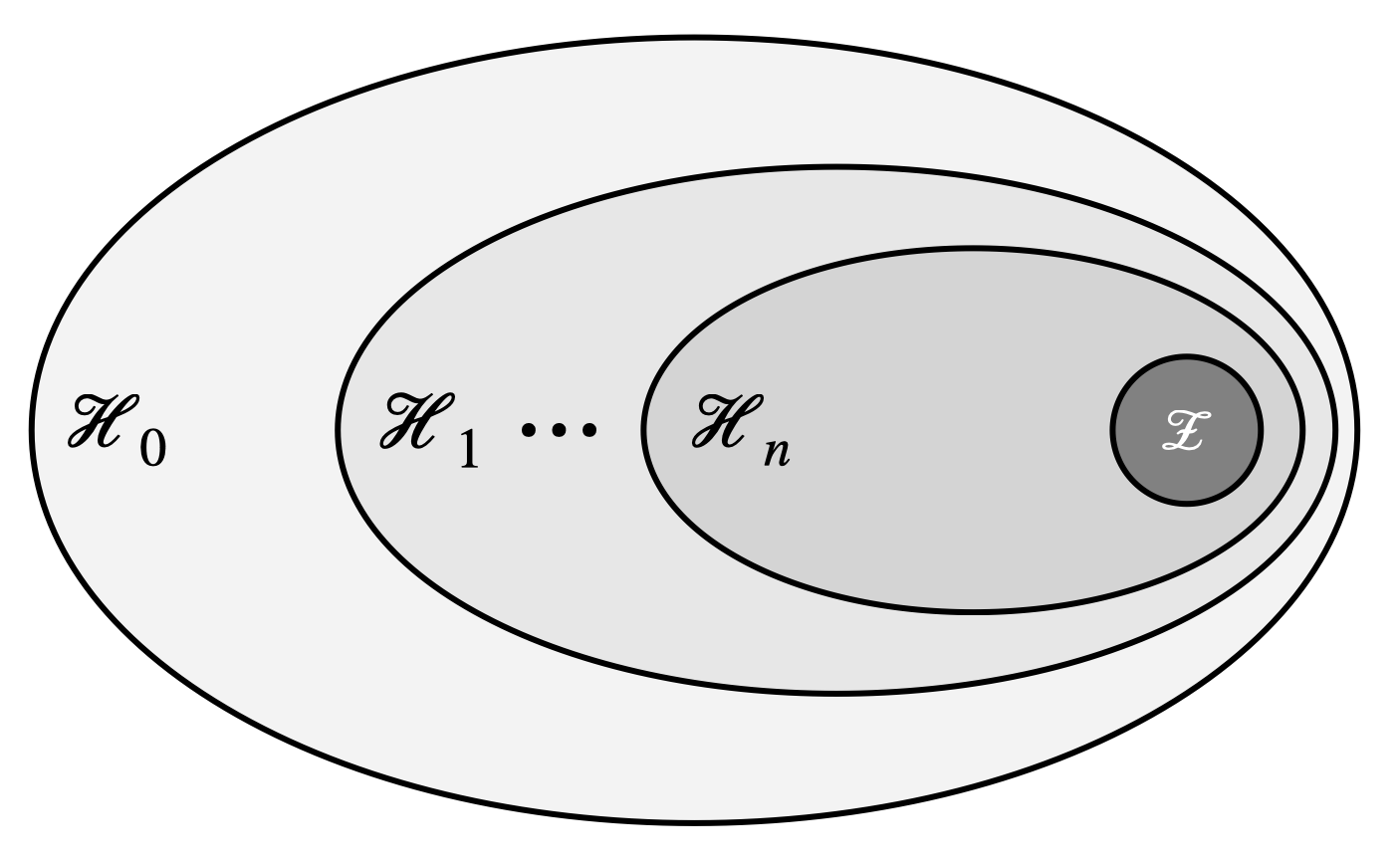}
\end{center}
where we identify each $\call{H}_i = \call{H}_{\pi_i}$ with a subspace of $\call{H}_{i-1}$ via $T_{\gamma_{i-1}}$ and $\call{Z}$ denotes the common space of ground states for the systems $\den{\pi_i}$. Note that even for the cut-free proof $\widetilde{\pi}$ we have that $\call{Z} \neq \call{H}_n$.

%\subsection{Context}

%In this note we present an assignment of quantum error-correcting codes to multiplicative proof nets, without any particular motivation. However, part of the justification for interest in this model is that it arises naturally in the context of Topological Quantum Field Theory (TQFT). There is an assignment to multiplicative proof nets of stratified $2$-manifolds, which when a TQFT based on Landau-Ginzburg models is applied to them and the cut operation \cite{murfet_cut} is used to evaluate them, leads precisely to the error-correcting codes studied here.

\section{Background}\label{section:background}

\subsection{Proofs}\label{sec:proof_nets}
\begin{defn}\label{def:formulas}
	There is an infinite set of \emph{unoriented atoms} $X,Y,Z,...$ and an \emph{oriented atom} (or \emph{atomic proposition}) is a pair $(X,+)$ or $(X,-)$ where $X$ is an unoriented atom. The set of \emph{pre-formulas} is defined as follows.
	\begin{itemize}
		\item Any atomic proposition is a pre-formula.
		\item If $A,B$ are pre-formulas then so are $A \otimes B$, $A \parr B$.
		\item If $A$ is a pre-formula then so is $\neg A$.
	\end{itemize}
	The set of \emph{formulas} is the quotient of the set of pre-formulas by the equivalence relation $\sim$ generated by, for arbitrary formulas $A,B$ and unoriented atom $X$, the following
	\[
		\neg (A \otimes B) \sim \neg A \parr \neg B,\qquad \neg (A \parr B) \sim \neg A \otimes \neg B,\qquad \neg (X, \pm) \sim (X, \mp)\,.
	\]
\end{defn}

\begin{defn}\label{def:proof_structures}
	A \emph{proof structure} is a directed multigraph with edges labelled with formulas (Definition \ref{def:formulas}) and with vertices labelled with an element from $\lbrace \ax, \cut, \otimes, \parr, \operatorname{c} \rbrace$. A proof structure may not admit any loops (however it may admit cycles). The incoming edges of a vertex are called its \emph{premisses}, the outgoing edges are its \emph{conclusions}. Proof structures are required to adhere to the following conditions.
	\begin{itemize}
		\item Each vertex labelled $\ax$ has exactly two conclusions and no premisse, the conclusions are labelled $A$ and $\neg A$ for some $A$.
		\item Each vertex labelled $\cut$ has exactly two premisses and no conclusion, where the premisses are labelled $A$ and $\neg A$ for some $A$. 
		\item Each vertex labelled $\otimes$ has exactly two premisses and one conclusion. These two premisses are ordered. The smallest one is called the \emph{left} premise of the vertex, the biggest one is called the \emph{right} premise. The left premise is labelled $A$, the right premise is labelled $B$ and the conclusion is labelled $A \otimes B$, for some $A,B$.
		\item Each vertex labelled $\parr$ has exactly two ordered premisses and one conclusion. The left premise is labelled $A$, the right premise is labelled $B$ and the conclusion is labelled $A \parr B$, for some $A,B$.
		\item Each vertex labelled $\operatorname{c}$ has exactly one premise and no conclusion. Such a premise of a vertex labelled $\operatorname{c}$ is called the \emph{conclusion} of the proof structure.
	\end{itemize}
	Let $\pi$ be a proof structure. A \emph{conclusion link} consists of a vertex labelled $\operatorname{c}$ along with its premise. An \emph{axiom link} of $\pi$ is a subgraph consisting of a vertex labelled $\ax$ along with its conclusions. A $\cut$ link consists of a vertex labelled $\cut$ along with its premises. A \emph{tensor link} of $\pi$ consists of a vertex labelled $\otimes$ along with its premises and conclusion. A \emph{par link} consists of a vertex labelled $\parr$ along with its premises and conclusion.
	% https://q.uiver.app/?q=WzAsMjcsWzIsMCwiXFxheCJdLFsxLDEsIlxcbmVnIEEiXSxbMywxLCJBIl0sWzMsMiwiXFx2ZG90cyJdLFsxLDIsIlxcdmRvdHMiXSxbNSwyLCJcXGN1dCJdLFs0LDEsIlxcbmVnIEEiXSxbNiwxLCJBIl0sWzQsMCwiXFx2ZG90cyJdLFs2LDAsIlxcdmRvdHMiXSxbMSwzLCJcXHZkb3RzIl0sWzMsMywiXFx2ZG90cyJdLFsxLDQsIkEiXSxbMyw0LCJCIl0sWzIsNSwiXFxvdGltZXMiXSxbMiw2LCJBIFxcb3RpbWVzIEIiXSxbMiw3LCJcXHZkb3RzIl0sWzQsMywiXFx2ZG90cyJdLFs2LDMsIlxcdmRvdHMiXSxbNCw0LCJBIl0sWzYsNCwiQiJdLFs1LDUsIlxccGFyciJdLFs1LDYsIkEgXFxwYXJyIEIiXSxbNSw3LCJcXHZkb3RzIl0sWzAsMCwiXFx2ZG90cyJdLFswLDEsIkEiXSxbMCwyLCJcXG9wZXJhdG9ybmFtZXtjfSJdLFswLDEsIiIsMCx7ImN1cnZlIjoyLCJzdHlsZSI6eyJoZWFkIjp7Im5hbWUiOiJub25lIn19fV0sWzAsMiwiIiwyLHsiY3VydmUiOi0yLCJzdHlsZSI6eyJoZWFkIjp7Im5hbWUiOiJub25lIn19fV0sWzEsNF0sWzIsM10sWzEwLDEyLCIiLDAseyJzdHlsZSI6eyJoZWFkIjp7Im5hbWUiOiJub25lIn19fV0sWzEyLDE0LCIiLDAseyJjdXJ2ZSI6Mn1dLFsxMywxNCwiIiwyLHsiY3VydmUiOi0yfV0sWzExLDEzLCIiLDIseyJzdHlsZSI6eyJoZWFkIjp7Im5hbWUiOiJub25lIn19fV0sWzE0LDE1LCIiLDIseyJzdHlsZSI6eyJoZWFkIjp7Im5hbWUiOiJub25lIn19fV0sWzE1LDE2XSxbMTcsMTksIiIsMix7InN0eWxlIjp7ImhlYWQiOnsibmFtZSI6Im5vbmUifX19XSxbMTgsMjAsIiIsMix7InN0eWxlIjp7ImhlYWQiOnsibmFtZSI6Im5vbmUifX19XSxbMTksMjEsIiIsMix7ImN1cnZlIjoyfV0sWzIwLDIxLCIiLDEseyJjdXJ2ZSI6LTJ9XSxbMjEsMjIsIiIsMSx7InN0eWxlIjp7ImhlYWQiOnsibmFtZSI6Im5vbmUifX19XSxbMjIsMjNdLFsyNCwyNSwiIiwxLHsic3R5bGUiOnsiaGVhZCI6eyJuYW1lIjoibm9uZSJ9fX1dLFsyNSwyNl0sWzgsNiwiIiwwLHsic3R5bGUiOnsiaGVhZCI6eyJuYW1lIjoibm9uZSJ9fX1dLFs5LDcsIiIsMCx7InN0eWxlIjp7ImhlYWQiOnsibmFtZSI6Im5vbmUifX19XSxbNyw1LCIiLDAseyJjdXJ2ZSI6LTJ9XSxbNiw1LCIiLDEseyJjdXJ2ZSI6Mn1dXQ==
	\[\begin{tikzcd}[column sep = small, row sep = small]
		\vdots && \ax && \vdots && \vdots \\
		A & {\neg A} && A & {\neg A} && A \\
		{\operatorname{c}} & \vdots && \vdots && \cut \\
		& \vdots && \vdots & \vdots && \vdots \\
		& A && B & A && B \\
		&& \otimes &&& \parr \\
		&& {A \otimes B} &&& {A \parr B} \\
		&& \vdots &&& \vdots
		\arrow[curve={height=12pt}, no head, from=1-3, to=2-2]
		\arrow[curve={height=-12pt}, no head, from=1-3, to=2-4]
		\arrow[from=2-2, to=3-2]
		\arrow[from=2-4, to=3-4]
		\arrow[no head, from=4-2, to=5-2]
		\arrow[curve={height=12pt}, from=5-2, to=6-3]
		\arrow[curve={height=-12pt}, from=5-4, to=6-3]
		\arrow[no head, from=4-4, to=5-4]
		\arrow[no head, from=6-3, to=7-3]
		\arrow[from=7-3, to=8-3]
		\arrow[no head, from=4-5, to=5-5]
		\arrow[no head, from=4-7, to=5-7]
		\arrow[curve={height=12pt}, from=5-5, to=6-6]
		\arrow[curve={height=-12pt}, from=5-7, to=6-6]
		\arrow[no head, from=6-6, to=7-6]
		\arrow[from=7-6, to=8-6]
		\arrow[no head, from=1-1, to=2-1]
		\arrow[from=2-1, to=3-1]
		\arrow[no head, from=1-5, to=2-5]
		\arrow[no head, from=1-7, to=2-7]
		\arrow[curve={height=-12pt}, from=2-7, to=3-6]
		\arrow[curve={height=12pt}, from=2-5, to=3-6]
	\end{tikzcd}\]
\end{defn}

Abusing notation slightly, we will conflate links with the vertices of the proof structure.

\begin{defn}
	A subgraph of a proof structure $\pi$ of the following form is an \emph{$a$-redex}.
	\begin{center}
		\begin{tabular}{ c c c }
			$
			\begin{tikzcd}[column sep = small, row sep = small]
				\vdots\arrow[d,dash] &&& \ax\arrow[dl, bend right, dash]\arrow[dr,dash, bend left]\\
				A\arrow[dr, bend right] && \neg A\arrow[dl,bend left] && A\arrow[d]\\
				& \cut &&& \vdots
			\end{tikzcd}
			$
			&
			\tagarray{\label{eq:a_redex}}
		\end{tabular}
	\end{center}
	A subgraph of a proof structure of the following form is an \emph{$m$-redex}.
	\begin{center}
		\begin{tabular}{ c c }
			$
			% https://q.uiver.app/?q=WzAsMTQsWzAsMF0sWzIsMSwiQSJdLFszLDIsIlxcb3RpbWVzIl0sWzQsMSwiQiJdLFs2LDAsInZfMyJdLFs2LDEsIlxcbmVnIEEiXSxbNywyLCJcXHBhcnIiXSxbOCwxLCJcXG5lZyBCIl0sWzgsMCwidl80Il0sWzMsMywiQSBcXG90aW1lcyBCIl0sWzcsMywiXFxuZWcgQSBcXHBhcnIgXFxuZWcgQiJdLFs1LDQsIlxcY3V0Il0sWzIsMCwidl8xIl0sWzQsMCwidl8yIl0sWzEsMiwiIiwwLHsiY3VydmUiOjJ9XSxbMywyLCIiLDAseyJjdXJ2ZSI6LTJ9XSxbNSw2LCIiLDAseyJjdXJ2ZSI6Mn1dLFs3LDYsIiIsMCx7ImN1cnZlIjotMn1dLFs0LDUsIiIsMix7InN0eWxlIjp7ImhlYWQiOnsibmFtZSI6Im5vbmUifX19XSxbOCw3LCIiLDAseyJzdHlsZSI6eyJoZWFkIjp7Im5hbWUiOiJub25lIn19fV0sWzYsMTAsIiIsMCx7InN0eWxlIjp7ImhlYWQiOnsibmFtZSI6Im5vbmUifX19XSxbOSwxMSwiIiwwLHsiY3VydmUiOjJ9XSxbMTAsMTEsIiIsMSx7ImN1cnZlIjotMn1dLFsyLDksIiIsMSx7InN0eWxlIjp7ImhlYWQiOnsibmFtZSI6Im5vbmUifX19XSxbMTIsMSwiIiwwLHsic3R5bGUiOnsiaGVhZCI6eyJuYW1lIjoibm9uZSJ9fX1dLFsxMywzLCIiLDAseyJzdHlsZSI6eyJoZWFkIjp7Im5hbWUiOiJub25lIn19fV1d
			\begin{tikzcd}[column sep = small, row sep = small]
				{} && {\vdots} && {\vdots} && {\vdots} && {\vdots} \\
				&& A && B && {\neg A} && {\neg B} \\
				&&& \otimes &&&& \parr \\
				&&& {A \otimes B} &&&& {\neg A \parr \neg B} \\
				&&&&& \cut
				\arrow[curve={height=12pt}, from=2-3, to=3-4]
				\arrow[curve={height=-12pt}, from=2-5, to=3-4]
				\arrow[curve={height=12pt}, from=2-7, to=3-8]
				\arrow[curve={height=-12pt}, from=2-9, to=3-8]
				\arrow[no head, from=1-7, to=2-7]
				\arrow[no head, from=1-9, to=2-9]
				\arrow[no head, from=3-8, to=4-8]
				\arrow[curve={height=12pt}, from=4-4, to=5-6]
				\arrow[curve={height=-12pt}, from=4-8, to=5-6]
				\arrow[no head, from=3-4, to=4-4]
				\arrow[no head, from=1-3, to=2-3]
				\arrow[no head, from=1-5, to=2-5]
			\end{tikzcd}
			$
			&
			\tagarray{\label{eq:m_redex}}
		\end{tabular}
	\end{center}
\end{defn}

\begin{defn}\label{def:reduction}
	Given a multiplicative linear logic proof structure $\pi$ admitting an $a$-redex $\zeta$ of the form given on the left in \eqref{rule:ax_reduct}, the reduction of $\pi$ is the proof $\pi'$ given by replacing the subgraph of $\pi$ on the left by what is displayed on the right in \eqref{rule:ax_reduct}.
	\begin{center}
		\begin{tabular}{ c c c }
			$
			\begin{tikzcd}[column sep = small, row sep = small]
				\vdots\arrow[d,dash] &&& \ax\arrow[dl, bend right, dash]\arrow[dr,dash, bend left]\\
				A\arrow[dr, bend right] && \neg A\arrow[dl,bend left] && A\arrow[d]\\
				& \cut &&& \vdots
			\end{tikzcd}
			$
			&
			$
			\begin{tikzcd}[column sep = small, row sep = small]
				\vdots\arrow[d,dash]\\
				A\arrow[d]\\
				\vdots
			\end{tikzcd}
			$
			&
			\tagarray{\label{rule:ax_reduct2}}
		\end{tabular}
	\end{center}
	Similarly if $\pi$ admits an $m$-redex, the associated reduction of $\pi$ is the proof $\pi'$ given by replacing the the subgraph of $\pi$ above by the that below in \eqref{rule:tensor_reduct}:
	\begin{center}
		\begin{tabular}{ c c }
			$\begin{tikzcd}[column sep = small, row sep = small]
				\vdots && \vdots && \vdots && \vdots \\
				A && B && {\neg A} && {\neg B} \\
				& \otimes &&&& \parr \\
				& {A \otimes B} &&&& {\neg A \parr \neg B} \\
				&&& \cut
				%			\arrow[curve={height=12pt}, from=4-2, to=5-4]
				%			\arrow[curve={height=-12pt}, from=4-6, to=5-4]
				\arrow[from=2-5, to=3-6]
				\arrow[from=2-7, to=3-6]
				\arrow[from=3-6, to=4-6, dash]
				\arrow[from=3-2, to=4-2, dash]
				\arrow[from=2-1, to=3-2]
				\arrow[from=2-3, to=3-2]
				\arrow[from=1-1, to=2-1, dash]
				\arrow[from=1-3, to=2-3, dash]
				\arrow[from=1-5, to=2-5, dash]
				\arrow[from=1-7, to=2-7, dash]
				\arrow[from=4-2, to=5-4, bend right]
				\arrow[from=4-6, to=5-4, bend left]
			\end{tikzcd}$\\
			& \tagarray{\label{rule:tensor_reduct}}\\
			$\begin{tikzcd}[column sep = small, row sep = small]
				\vdots && \vdots && \vdots && \vdots \\
				A && B && {\neg A} && {\neg B} \\
				& \cut &&&& \cut
				\arrow[from=1-1, to=2-1, dash]
				\arrow[from=1-3, to=2-3, dash]
				\arrow[from=1-5, to=2-5, dash]
				\arrow[from=1-7, to=2-7, dash]
				\arrow[from=2-1, to=3-2, bend right]
				\arrow[from=2-5, to=3-2, bend left]
				\arrow[from=2-3, to=3-6, bend right]
				\arrow[from=2-7, to=3-6, bend left]
			\end{tikzcd}$
		\end{tabular}
	\end{center}
	A \emph{reduction} $\gamma: \pi \lto \pi'$ is a pair of proof structures $(\pi,\pi')$ where $\pi'$ is the result of applying one of the reduction rules just described to $\pi$. We write $\pi \lto_{\cut} \pi'$ when $(\pi,\pi')$ is a reduction. A \emph{reduction sequence} $\Gamma: \pi \lto \pi'$ is a nonempty sequence of reductions.
\end{defn}

%\begin{remark} Remark on MELL, strength, etc. This is a limited system.
%\end{remark}

\subsection{Exterior algebras}\label{section:exterior_alg}

Associated to a vector space $V$ over the complex numbers $\bb{C}$ are three canonical $\bb{C}$-algebras: the tensor algebra (which is universal among not-necessarily commutative algebras with a map from $V$), the symmetric algebra $\operatorname{Sym}(V)$ (which is universal among all commutative algebras with a map from $V$) and the exterior algebra $\bigwedge V$ (see below for the universal property). All our algebras are associative and unital. These examples of algebras are all $\mathbb{Z}$-graded, in the sense that vectors may be written uniquely as a sum of \emph{homogeneous} vectors of degrees $d \in \mathbb{Z}$ where the degree is additive with respect to multiplication. The subspace of elements of degree $d$ in the symmetric and exterior algebras are denoted $\operatorname{Sym}^k(V)$ and $\bigwedge^k V$ respectively.

The exterior algebra
\begin{equation}\label{eq:wpofgfcvs}
	\bigwedge (\bb{C}\psi_1 \oplus \ldots \oplus \bb{C}\psi_n)
	\end{equation}
of the vector space $\bb{C}\underline{\psi} = \bb{C}\psi_1 \oplus \ldots \oplus \bb{C}\psi_n$ has a $\bb{C}$-basis consisting of wedge products
\begin{equation}\label{eq:exterior_basis}
	\psi_{i_1} \wedge \cdots \wedge \psi_{i_r} \qquad 1 \le i_1 < \cdots < i_r \le n
\end{equation}
with the empty wedge product being the identity element of the algebra. Here $\wedge$ denotes the multiplication in the algebra, which is associative and unital but not commutative. The fact that the wedge products \eqref{eq:exterior_basis} form a basis may be deduced from the fact that for vector spaces $V, W$ there is a natural isomorphism
\begin{equation}\label{eq:bigwedge_is_exp}
	\bigwedge( V \oplus W ) \cong \bigwedge V \otimes \bigwedge W
\end{equation}
and the fact that
\begin{equation}\label{eq:qubit}
	\bigwedge( \bb{C} \psi ) \cong \bb{C}1 \oplus \bb{C} \psi \cong \mathbb{C}^2
\end{equation}
by which we mean that the exterior algebra on a one-dimensional vector space has a basis consisting of the empty wedge product $1$ and the singleton wedge product $\psi$. The exterior algebra is naturally $\mathbb{Z}$-graded, where we give a wedge product \eqref{eq:exterior_basis} degree $r$. By definition $\psi \wedge \varphi = -\varphi \wedge \psi = 0$ for any $\psi, \varphi \in \bb{C} \underline{\psi}$ and so in particular $\psi \wedge \psi = 0$.

One of the universal properties of the exterior algebra involves $\mathbb{Z}_2$-graded algebras: a $\mathbb{Z}_2$-grading on an algebra $C$ is a decomposition $C = C_0 \oplus C_1$ with $1 \in C_0$ and $xy = (-1)^{|x||y|} yx$ whenever $x,y$ are homogeneous. We call the vectors $x \in C_0$ \emph{even} and the vectors $y \in C_1$ \emph{odd}. The exterior algebra has an obvious $\mathbb{Z}_2$-grading obtained from its natural $\mathbb{Z}$-grading
\begin{equation}
	\bigwedge V = \Big[ \bigoplus_{i \in 2 \mathbb{N}} \bigwedge^i V \Big] \oplus  \Big[ \bigoplus_{i \in 2 \mathbb{N} + 1} \bigwedge^{i} V \Big]\,.
\end{equation}
A \emph{morphism} $f: B \lto C$ of $\mathbb{Z}_2$-graded algebras is a morphism of algebras with $f(x) \in C_i$ for all $x \in C_i$, where $i \in \{0,1\}$.

\begin{proposition}\label{prop:univ_prop_exterior} If $\theta: V \lto C$ is a linear map where $C$ is a $\mathbb{Z}_2$-graded algebra and $\alpha(v)$ is odd for all $v \in V$, then there is a unique morphism of $\mathbb{Z}_2$-graded algebras $\Theta: \bigwedge V \lto C$ making the diagram
\begin{equation}
\xymatrix@C+2pc{
	V \ar[r]\ar[dr]_-{\theta} & \bigwedge V \ar@{.>}[d]^{\Theta}\\
	& C
}
\end{equation}
commute, where the horizontal map is the canonical inclusion.
\end{proposition}
\begin{proof}
The universal property of the tensor algebra $TV$ gives a morphism of algebras $TV \lto C$, and it is easily checked that it vanishes on the ideal that defines the exterior algebra as a quotient of the tensor algebra.
\end{proof}

An arbitrary linear operator on the exterior algebra \eqref{eq:wpofgfcvs} can be described as linear combinations of products of \emph{wedge product} and \emph{contraction} operators, which we now describe. Given $\varphi \in \bb{C}\underline{\psi}$ the operator of left multiplication
\begin{gather*}
	\varphi \wedge (-): \bigwedge \bb{C} \underline{\psi} \lto \bigwedge\bb{C} \underline{\psi}\\
	\psi_{i_1} \wedge \hdots \wedge \psi_{i_r} \longmapsto \varphi \wedge \psi_{i_1} \wedge \hdots \wedge \psi_{i_r}
\end{gather*}
increases the $\mathbb{Z}$-degree by one. Given a dual vector $\eta \in (\bb{C} \underline{\psi})^*$ the operation of contraction
\begin{gather*}
	\eta \lrcorner (-): \bigwedge \bb{C}\underline{\psi} \lto \bigwedge \bb{C}\underline{\psi}\\
	\psi_{i_1} \wedge \hdots \wedge \psi_{i_r} \lto \sum_{j = 1}^r (-1)^{j-1}\eta(\psi_{i_j})\psi_{i_1} \wedge \hdots \wedge \hat{\psi}_{i_j} \wedge \hdots \wedge \psi_{i_r}
\end{gather*}
is $\bb{C}$-linear and decreases the $\mathbb{Z}$-degree by one.

To each generator $\psi_i$ is associated both a wedge product operator $\psi_i \wedge (-)$ and also a contraction operator $(\psi_i)^* \lrcorner (-)$ where $(\psi_i)^*$ is defined by $(\psi_i)^*(\psi_j) = \delta_{ij}$. Where it will not cause confusion we conflate these two operators with $\psi_i$ and $\psi_i^*$ themselves, writing
\begin{equation}
\psi_i = \psi_i \wedge (-)\,, \qquad \psi_i^* = (\psi_i)^* \lrcorner (-)\,.
\end{equation}
These operators satisfy the canonical anti-commutation relations
\begin{align}
\psi_i \psi_j + \psi_j \psi_i &= 0 & 1 \le i,j \le n\\
\psi_i^* \psi_j^* + \psi_j^* \psi_i^* &= 0 & 1 \le i,j \le n\\
\psi_i \psi_j^* + \psi_j^* \psi_i &= \delta_{ij} & 1 \le i,j \le n
\end{align}
If we define $\langle \psi_i, \psi_j \rangle = \delta_{ij}$ then $\bb{C}\underline{\psi}$ is a Hilbert space, and this extends canonically to a Hilbert space structure on the exterior algebra which is completely described by basis elements $\psi_{i_1} \wedge \cdots \wedge \psi_{i_r}$ and $\psi_{j_1} \wedge \cdots \wedge \psi_{j_s}$ of different lengths $r \neq s$ being orthogonal, and if $r = s$ then $\langle \psi_{i_1} \wedge \cdots \wedge \psi_{i_r}\,, \psi_{j_1} \wedge \cdots \wedge \psi_{j_s} \rangle = \delta_{i_1 = j_1} \cdots \delta_{i_r = j_s}$. 

When we refer to $\bigwedge \bb{C}\underline{\psi}$ as a Hilbert space, it is always this pairing that is intended, so that the wedge products \eqref{eq:exterior_basis} are an orthonormal basis. Since the exterior algebra has a Hilbert space structure, we can talk of adjoints:

\begin{lemma} The operator $\psi_i$ is adjoint to $\psi_i^*$.
\end{lemma}
\begin{proof}
	It suffices to prove that
	\[
		\Big\langle \psi_i\Big( \psi_{i_1} \wedge \cdots \wedge \psi_{i_r} \Big) \,, \psi_{j_1} \wedge \cdots \wedge \psi_{j_s} \Big\rangle = \Big\langle \psi_{i_1} \wedge \cdots \wedge \psi_{i_r} \,, \psi_i^*\Big( \psi_{j_1} \wedge \cdots \wedge \psi_{j_s} \Big) \Big\rangle
	\]
	for all appropriate pairs of sequences. The left hand side is nonzero if and only if $i \notin \{ i_1, \ldots, i_r \}$, $r + 1 = s$ and $\{ i \} \cup \{ i_1, \ldots, i_r \} = \{ j_1, \ldots, j_s \}$ in which case it is either $+1$ or $-1$ depending on whether the position of $i$ in the former set, when it is arranged in increasing order, is odd or even. The right hand side is nonzero if and only if $i \in \{ j_1, \ldots, j_r \}$, $r = s - 1$ and $\{ j_1, \ldots, j_s \} \setminus \{ i \} = \{ i, i_1, \ldots, i_r \}$, in which case it is either $+1$ or $-1$ depending on whether the position of $i$ in the set of $j$'s is odd or even. It is easy to see these are the same.
\end{proof}

If we denote by $(-)^\dagger$ the adjoint of an operator, then $( \psi_i \wedge - )^\dagger = (\psi_i)^* \lrcorner (-)$ or somewhat more confusingly $\psi_i^\dagger = \psi_i^*, (\psi_i^*)^\dagger = \psi_i$.

\begin{remark} In physics the exterior algebra is called the space of \emph{Grassmann numbers}, and it arises naturally as follows: if $\call{H}$ is the Hilbert space of a single particle, then the space of states of $k$-particles is $\operatorname{Sym}^k(\call{H})$ for bosons and $\bigwedge^k \call{H}$ for fermions. This follows from the universal property and the hypothesis that wave functions of bosons are invariant to permutation, while those of fermions change sign. As a Hilbert space, $\bigwedge \call{H}$ is sometimes called \emph{fermionic Fock space}, $\psi_i$ a \emph{creation} operator and $\psi_i^*$ an \emph{annihilation} operator.
\end{remark}

\subsection{Mathematics of Quantum Mechanics}

A \emph{qubit} is the Hilbert space $\mathbb{C}^2$ with standard orthonormal basis denoted $\ket{0}, \ket{1}$. Another orthonormal basis is
\begin{equation}\label{eq:plus_minus}
	\ket{+} = \tfrac{1}{\sqrt{2}}( \ket{0} + \ket{1} )\,, \qquad \ket{-} = \tfrac{1}{\sqrt{2}}( \ket{0} - \ket{1} )\,.
\end{equation}
The Pauli sigma matrices are the linear operators (in the standard basis)
\begin{equation}
	X = \begin{pmatrix} 0 & 1 \\ 1 & 0 \end{pmatrix}\,, \qquad Z = \begin{pmatrix} 1 & 0 \\ 0 & -1 \end{pmatrix}\,, \qquad Y = \begin{pmatrix} 0 & -i \\ i & 0 \end{pmatrix} = iXZ\,.
	\end{equation}
Some basic facts that can be read off from the above:
\begin{itemize}
\item $X,Z$ anticommute, that is $XZ = - ZX$.
\item $X \ket{+} = \ket{+}$ and $X \ket{-} = -\ket{-}$. The spectrum of $X$ is $\{-1,+1\}$.
\item $Z \ket{+} = \ket{-}$ and $Z \ket{-} = \ket{+}$. The spectrum of $Z$ is also $\{-1,+1\}$.
\item $X,Y,Z$ are self-adjoint, involutive (square to the identity) and unitary.
\end{itemize}
The Hilbert spaces of quantum computation are tensor products of single qubit Hilbert spaces $\call{H} = (\mathbb{C}^2)^{\otimes n}$. Following physics notation we denote vectors in these tensor products by kets, for example
\begin{align*}
\ket{001} &= \ket{0} \otimes \ket{0} \otimes \ket{1} \in (\mathbb{C}^2)^{\otimes 3}\\
\ket{+-+} &= \ket{+} \otimes \ket{-} \otimes \ket{+} \in (\mathbb{C}^2)^{\otimes 3}\,.
\end{align*}
If $F$ is an operator on $\mathbb{C}^2$ (we sometimes call this a \emph{single qubit operator}) then $1 \otimes \cdots \otimes F \otimes \cdots \otimes 1$ acting on $(\mathbb{C}^2)^{\otimes n}$ with the $F$ in the $i$th tensor factor will be denoted $F_i$. Note that for single-qubit operators $F, F'$ the operators $F_i, F'_j$ on $(\mathbb{C}^2)^{\otimes n}$ commute if $i \neq j$.

\begin{defn} The \emph{Pauli group} $G_n$ on $n$ qubits is the group of invertible linear operators on the tensor product Hilbert space $(\mathbb{C}^2)^{\otimes n}$ generated by $X,Y,Z,-1,i$ operating in each qubit \cite[\S 10.5.1]{nc}. This is a subgroup of the group of unitary operators.
\end{defn}

\begin{example}\label{example:three_qubit} Consider the operators
\[
X_1 X_2 = X \otimes X \otimes 1\,, \qquad X_2 X_3 = 1 \otimes X \otimes X
\]
on $\call{H} = (\mathbb{C}^2)^{\otimes 3}$. These are elements of the Pauli group $G_3$. Note that for $a,b,c \in \{-,+\}$
\begin{align}
X_1 X_2 \ket{abc} &= ab \ket{abc}\,,\label{eq:abc_ket}\\
X_2 X_3 \ket{abc} &= bc \ket{abc}\label{eq:abc_ket2}
\end{align}
where the coefficients are read as $\pm 1$. Note that for $a \in \{0,1\}$ the state
\begin{equation}\label{eq:entangledstate}
\ket{\tilde{a}} = \tfrac{1}{\sqrt{2}}\big( \ket{+++} + (-1)^a \ket{---})
\end{equation}
is a $+1$-eigenvector of $X_1 X_2, X_2 X_3$ and in fact these two states form a basis for the $+1$-eigenspace. This is an example of an \emph{entangled state} (see Remark \ref{remark:entangled_measure} below). The self-adjoint operator $H = -X_1X_2 - X_2X_3$ has spectrum $\{-2,-1,0,1,2\}$ and its $(-2)$-eigenspace (the space of ground states) is equal to the joint $+1$-eigenspace of $X_1X_2, X_2X_3$ and is therefore also spanned by the states in \eqref{eq:entangledstate}.
\end{example}

\begin{remark}\label{remark:entangled_measure}
	Let us explain what it means for the state $\ket{\tilde{a}}$ of \eqref{eq:entangledstate} to be entangled. Consider three observers, each with exclusive access to one of the three qubits $1,2,3$ in  Example \ref{example:three_qubit}. If the first observer measures their qubit in the basis $\ket{+}, \ket{-}$ then with probability $0.5$ they measure $\ket{+}$ and the state is projected onto $\ket{++}$ for the other two observers (both of whom will measure $\ket{+}$ with probability $1$ if they do the corresponding measurement on their qubit) and with probability $0.5$ the $0.5$ the first observer measures $\ket{-}$ and the state is projected onto $\ket{--}$ (the remaining two observers both measuring $\ket{-}$ with probability $1$). In this operational sense, the ``individual'' states of the three qubits cannot be distinguished by measurements when the joint state is entangled.
	\end{remark}
	
\subsection{Quantum Error-Correction}

Physical information processing systems work by propagating messages, and these must be protected against the effects of noise. This is done by adding redundant information to the message, or \emph{encoding} it, in such a way that even if some of the information in the message is corrupted by noise, the original message can be reconstructed. The specification of \emph{how} to encode and reconstruct messages is the data of an \emph{error-correcting code}. In the case that the messages are quantum states, these are \emph{quantum error-correcting codes} \cite{nc}.

Our recommended references for quantum error-correction are \cite{nc}. Here we give a very brief review. For succinctness in this paper we appreviate {quantum error-correcting code} to \emph{quantum code}.

\begin{defn}\label{defn:stabiliser_code} Let $\call{H}$ be a finite-dimensional Hilbert space. A \emph{stabiliser quantum code} on $\call{H}$ is a commutative subgroup $S \subseteq U(\call{H})$ of the group of unitary operators on $\call{H}$ with $-1 \notin S$ such that there exists an integer $n > 1$ and isomorphism of Hilbert spaces
\[
\Phi: \call{H} \lto (\mathbb{C}^2)^{\otimes n}
\]
where for every $g \in S$ the operator $\Phi g \Phi^{-1}$ belongs to the Pauli group $G_n$. We refer to $S$ as the \emph{stabiliser group} of the code and invariant vectors as \emph{codewords}.
\end{defn}

%A \emph{stabiliser Quantum Error-Correcting Code} (QECC) on a Hilbert space $\call{H} = (\mathbb{C}^2)^{\otimes n}$ is a subgroup $S = \langle g_1, \ldots, g_r \rangle$ of the Pauli group $G_n$. The generators $g_1,\ldots,g_r$ are referred to as \emph{stabilisers}. The subspace $\call{H}_S$ of vectors invariant to every $g \in S$ is called the \emph{code subspace} or \emph{space of codewords}.

\begin{remark}\label{remark:stab_prop} We refer to \cite[\S 10.5.1]{nc} for the basic theory of stabiliser codes. In particular, observe that since every pair of operators in the Pauli group either commutes or anticommutes, the conditions that $S$ is commutative and $-1 \notin S$ are necessary for there to be nonzero codewords. It is also easy to see that if $S$ is a stabiliser code then $g^2 = 1$ for every $g \in S$ and every operator in $S$ is self-adjoint.
\end{remark}

\begin{example} The stabiliser group $S = \langle X_1 X_2, X_2 X_3 \rangle$ determines the three qubit phase flip code \cite[\S 10.5.6]{nc}.
\end{example}

\begin{remark}
In the analogy with the transmission of messages, we have a space $\call{H}$ of possible messages and a subspace $\call{C}$ of codewords: encoded messages are always codewords. That means that if we receive a message $h \in \call{H} \setminus \call{C}$ we know an error has occurred; detecting this fact and replacing our message by the closest codeword $c = P(h)$ is the process of error-correction. From the point of view of geometry of interaction it is interesting that, in cases where the error can be unambiguously corrected, the vectors $h, c$ have the same \emph{denotation} (the intended message) but different \emph{senses}. 
\end{remark}

%\begin{remark} Something about how code spaces are highly entangled states, generically.
%\end{remark}

\subsection{Exterior algebra and qubits}\label{section:ext_qubit}

Recall from \eqref{eq:qubit} that the exterior algebra $\bigwedge( \bb{C} \psi )$ can be identified with the single qubit Hilbert space, taking $\ket{0} = 1, \ket{1} = \psi$. Under this identification the Pauli operators $X,Z$ correspond respectively to $\psi + \psi^*, \psi^* \psi - \psi \psi^*$. More generally, with $\bb{C}\underline{\psi} = \mathbb{C} \psi_1 \oplus \cdots \oplus \mathbb{C} \psi_n$ we have an isomorphism of Hilbert spaces
\begin{gather}
(\mathbb{C}^2)^{\otimes n} \cong \bigwedge \mathbb{C} \underline{\psi} \label{eq:iso_sneak}\\
\ket{a_1 \,\cdots\, a_n} \longmapsto \psi_1^{a_1} \wedge \cdots \wedge \psi_n^{a_n} \label{eq:iso_sneak_basis}
\end{gather}
where $a_1,\ldots,a_n \in \{0,1\}$ and $\psi^0 = 1, \psi^1 = \psi$. The left hand side has its usual Hilbert space structure, while the right hand side has the structure defined in Section \ref{section:exterior_alg}. There is a simple relationship between the operators $\{ \psi_i, \psi_i^* \}_{i=1}^n$ on the right hand side and the operators $\{ X_i, Z_i \}_{i=1}^n$ on the left hand side, once we account for the signs. Note that
\begin{align*}
X_i \ket{a_1 \,\cdots\, a_n} &= \ket{a_1 \,\cdots\, \Bar{a_i} \,\cdots \, a_n}\,\\
Z_i \ket{a_1 \,\cdots\, a_n} &= (-1)^{a_i} \ket{a_1 \,\cdots\, a_n}\,.
\end{align*}
For mathematicians the next lemma says that the $Z$ operators behave like the ``Koszul signs'' implicit in applying $\psi_i,\psi_i^*$ to the $\mathbb{Z}_2$-graded vector space $\bigwedge \mathbb{C} \underline{\psi}$ while physicists are familiar with the same identities under the name of the Jordan-Wigner transformation.

\begin{lemma} As operators on \eqref{eq:iso_sneak} we have for $1 \le i \le n$
\begin{gather}
\psi_i + \psi_i^* = Z_1 \,\cdots\, Z_{i-1} X_i \label{eq:iso_sneak_1}\\
\psi_i^* \psi_i - \psi_i \psi_i^* = Z_i \label{eq:iso_sneak_2}\\
\psi_i - \psi_i^* = - Z_1 \,\cdots\, Z_{i} X_i \label{eq:iso_sneak_3}
\end{gather}
\end{lemma}
\begin{proof}
The identities can be checked on the basis \eqref{eq:iso_sneak_basis}.
\end{proof}

The algebra $\bigwedge \mathbb{C} \underline{\psi}$ is $\mathbb{Z}_2$-graded, and the \emph{grading operator} $G$ is the linear operator defined on a homogeneous element $\Psi$ by $G\Psi = (-1)^a \Psi$ where $a \in \{0,1\}$ is the $\mathbb{Z}_2$-degree.

\begin{lemma}\label{lemma:grading_operator} As operators on \eqref{eq:iso_sneak} we have $G = Z_1 \,\cdots\, Z_n$.
\end{lemma}
\begin{proof}
The operator $\psi_i^* \psi_i - \psi_i \psi_i^*$ is clearly the grading operator on $\bigwedge \mathbb{C} \psi_i$ and since it is even, it commutes with $\psi_j^* \psi_j - \psi_j \psi_j^*$ for any $j \neq i$. Hence the claim follows from \eqref{eq:iso_sneak_2}.
\end{proof}

\begin{lemma}\label{lemma:bell_states_homogeneous} For $a \in \{0,1\}$ the vectors
	\[
		\ket{\widetilde{a}} = \tfrac{1}{\sqrt{2}}\big( \ket{+ \,\cdots\, +} + (-1)^a \ket{-\, \cdots \,-})
	\]
	of \eqref{eq:iso_sneak} are mutually orthogonal, unit length, and homogeneous of degree $a$.
\end{lemma}
\begin{proof}
It is easy to check that $\langle \widetilde{b} \,|\, \widetilde{a} \rangle = \delta_{ab}$. Note that by Lemma \ref{lemma:grading_operator}
\begin{align*}
G \ket{\widetilde{a}} &= \tfrac{1}{\sqrt{2}}\big( G\ket{+ \,\cdots\, +} + (-1)^a G\ket{-\, \cdots \,-})\\
&= \tfrac{1}{\sqrt{2}}\big( Z_1 \,\cdots\, Z_n \ket{+ \,\cdots\, +} + (-1)^a Z_1 \,\cdots\, Z_n \ket{-\, \cdots \,-})\\
&= \tfrac{1}{\sqrt{2}}\big( Z_1 \,\cdots\, Z_n \ket{+ \,\cdots\, +} + (-1)^a Z_1 \,\cdots\, Z_n \ket{-\, \cdots \,-})\\
&= \tfrac{1}{\sqrt{2}}\big( \ket{- \,\cdots\, -} + (-1)^a \ket{+\, \cdots \,+})\\
&= (-1)^a \ket{\widetilde{a}}
\end{align*}
which shows that $\ket{\widetilde{a}}$ is homogeneous of degree $a$.
\end{proof}

\begin{remark} When $n = 2$ the vectors $\ket{\widetilde{a}}$ are called \emph{Bell states} and they are the most fundamental example of entangled states in a tensor product of Hilbert spaces.
\end{remark}

\section{The Code of a Proof Net}\label{sec:Hilbert_Space_model}

\subsection{The Hilbert Space}\label{section:hilbert_space}

In this section $\pi$ is a proof structure and $\call{L}$ the set of links in $\pi$. We will define a Hilbert space $\call{H}_\pi$ by associating a set of fermions with each link.

\begin{defn}[\textbf{Unoriented atoms of a link}]\label{def:unoriented_atoms}
	To each link $l$ in $\pi$ we associate a set of unoriented atoms $U_l$. If $l$ is a Conclusion link then $U_l = \varnothing$. For an Axiom/Cut link
	\begin{center}
		\begin{tabular}{c c}
			$
			\begin{tikzcd}[column sep = small, row sep = small]
				& \ax\arrow[dl,bend right, dash]\arrow[dr,bend left, dash]\\
				\neg A\arrow[d] & & A\arrow[d]\\
				\vdots & & \vdots
			\end{tikzcd}$
			&
			$
			\begin{tikzcd}[column sep = small, row sep = small]
				\vdots\arrow[d, dash] & & \vdots\arrow[d, dash]\\
				\neg A\arrow[dr,bend right] & & A\arrow[dl, bend left]\\
				& \cut
			\end{tikzcd}$
		\end{tabular}
	\end{center}
involving formulas $A, \neg A$ with the same set of unoriented atoms $\lbrace X_1,...,X_n\rbrace$ we define
\begin{equation*}
	U_l = \lbrace X_1,...,X_n\rbrace\,.
	\end{equation*}
	For $\otimes/\parr$ links involving formulas $A,B$
	\begin{center}
		\begin{tabular}{ c c }
			% https://q.uiver.app/?q=WzAsNyxbMCwwLCJcXHZkb3RzIl0sWzAsMSwiQSJdLFsyLDEsIkIiXSxbMSwyLCJcXG90aW1lcyJdLFsxLDMsIkEgXFxvdGltZXMgQiJdLFsxLDQsIlxcdmRvdHMiXSxbMiwwLCJcXHZkb3RzIl0sWzQsNV0sWzIsMywiIiwwLHsiY3VydmUiOi0yfV0sWzEsMywiIiwyLHsiY3VydmUiOjJ9XSxbMyw0LCIiLDIseyJzdHlsZSI6eyJoZWFkIjp7Im5hbWUiOiJub25lIn19fV0sWzYsMiwiIiwwLHsic3R5bGUiOnsiaGVhZCI6eyJuYW1lIjoibm9uZSJ9fX1dLFswLDEsIiIsMix7InN0eWxlIjp7ImhlYWQiOnsibmFtZSI6Im5vbmUifX19XV0=
			$\begin{tikzcd}[column sep = small, row sep = small]
				\vdots && \vdots \\
				A && B \\
				& \otimes \\
				& {A \otimes B} \\
				& \vdots
				\arrow[from=4-2, to=5-2]
				\arrow[curve={height=-12pt}, from=2-3, to=3-2]
				\arrow[curve={height=12pt}, from=2-1, to=3-2]
				\arrow[no head, from=3-2, to=4-2]
				\arrow[no head, from=1-3, to=2-3]
				\arrow[no head, from=1-1, to=2-1]
			\end{tikzcd}$
			&
			% https://q.uiver.app/?q=WzAsNyxbMCwwLCJcXHZkb3RzIl0sWzAsMSwiQSJdLFsyLDEsIkIiXSxbMSwyLCJcXHBhcnIiXSxbMSwzLCJBIFxccGFyciBCIl0sWzEsNCwiXFx2ZG90cyJdLFsyLDAsIlxcdmRvdHMiXSxbNCw1XSxbMiwzLCIiLDAseyJjdXJ2ZSI6LTJ9XSxbMSwzLCIiLDIseyJjdXJ2ZSI6Mn1dLFszLDQsIiIsMix7InN0eWxlIjp7ImhlYWQiOnsibmFtZSI6Im5vbmUifX19XSxbNiwyLCIiLDAseyJzdHlsZSI6eyJoZWFkIjp7Im5hbWUiOiJub25lIn19fV0sWzAsMSwiIiwyLHsic3R5bGUiOnsiaGVhZCI6eyJuYW1lIjoibm9uZSJ9fX1dXQ==
			$\begin{tikzcd}[column sep = small, row sep = small]
				\vdots && \vdots \\
				A && B \\
				& \parr \\
				& {A \parr B} \\
				& \vdots
				\arrow[from=4-2, to=5-2]
				\arrow[curve={height=-12pt}, from=2-3, to=3-2]
				\arrow[curve={height=12pt}, from=2-1, to=3-2]
				\arrow[no head, from=3-2, to=4-2]
				\arrow[no head, from=1-3, to=2-3]
				\arrow[no head, from=1-1, to=2-1]
			\end{tikzcd}$
		\end{tabular}
	\end{center}
where $A$ has unoriented atoms $\lbrace X_1,...,X_n\rbrace$ and $B$ has unoriented atoms $\lbrace Y_1,...,Y_m\rbrace$,
\begin{equation*}
	U_l = \lbrace X_1,...,X_n, Y_1,...,Y_m\rbrace\,.
	\end{equation*}
	\end{defn}

Informally, if $X$ is an unoriented atom in a formula $A$ on an edge incident at a non-conclusion link $l$, then there is ``another copy'' of $X$ somewhere on another edge incident at $l$ and the unoriented atom $X \in U_l$ stands for this pair.

\begin{defn} The \emph{set of qubits} $Q_\pi$ of $\pi$ is the disjoint union
	\begin{equation}
		Q_\pi = \coprod_{l \in \call{L}}U_l\,.
	\end{equation}
	A \emph{qubit ordering} for $\pi$ is a list $U_1,\ldots,U_n$ where $Q_\pi = \{ U_1, \ldots, U_n \}$.
	\end{defn} 

In \cite[Definition 3.16]{algpnt} we defined a set $U_\pi$ of unoriented atoms of $\pi$ by taking a disjoint union over \emph{edges}, of the set of unoriented atoms of the formula $A$ labelling each edge. In this paper our unoriented atoms are associated to \emph{links}, and to avoid confusion we tend to refer to them as the \emph{qubits} of $\pi$ rather than the ``unoriented atoms of $\pi$''. This duality of perspectives will be fully explored elsewhere but the next lemma makes the basic point.

Recall that there is an equivalence relation $\approx$ on $U_\pi$ generated by a relation $\sim$ \cite[Definition 4.8]{algpnt}. Two unoriented atoms $V, V'$ satisfy $V \sim V'$ if they occur in formulas on edges incident at a common non-conclusion link $l$ \cite[Definition 3.18, 3.19]{algpnt}.

\begin{lemma}\label{lemma:duality_links_edges} There is a bijection between the set of qubits $Q_\pi$ and the set of unordered pairs $V, V' \in U_\pi$ of unoriented atoms of $\pi$ with $V \sim V'$.
\end{lemma}
\begin{proof}
Given a pair $V, V' \in U_\pi$ with $V \sim V'$ by inspection of \cite[Definition 3.18, 3.19]{algpnt} there is a non-conclusion link $l$ with $V, V'$ determining the same unoriented atom in $U_l$. This defines a map from pairs $V \sim V'$ to qubits which is bijective by construction.
\end{proof}

\begin{defn}\label{def:objects}
	 The \emph{Hilbert space $\call{H}_l$ of a link $l \in \call{L}$} is
	\begin{equation}
		\call{H}_{l} = \bigwedge \bigoplus_{X \in U_l}\bb{C} \psi_X^l
		\end{equation}
	where $\psi^l_X$ is a formal generator corresponding to $X \in U_l$. The \emph{Hilbert space of $\pi$} is
	\begin{equation}\label{eq:defn_hpi}
		\call{H}_\pi = \bigwedge \bigoplus_{l \in \call{L}}\bigoplus_{X \in U_l}\bb{C}\psi_X^l\,.
		\end{equation}
	\end{defn}

We refer to the $\psi_X^l$ as \emph{link fermions}. Sometimes we denote $\psi^l_X$ by $\psi_X$, keeping in mind each generator is associated with a unique link. If we choose an ordering on the links then we get an isomorphism $\call{H}_\pi \cong \bigotimes_{l} \call{H}_l$. Our reason for caring about qubit orderings is that if we want to represent $\call{H}$ as a qubit Hilbert space, we need to choose an ordering:

\begin{defn}\label{defn:iso_gamma} Let $U_1,\ldots,U_n$ be a qubit ordering for $\pi$. The associated isomorphism of Hilbert spaces $\Gamma: (\mathbb{C}^2)^{\otimes n} \lto \call{H}_\pi$ defined by \eqref{eq:bigwedge_is_exp} and \eqref{eq:iso_sneak} is
\begin{equation}
\Gamma \ket{a_1 \,\cdots\, a_n} = \psi_{U_1}^{a_1} \wedge \cdots \wedge \psi_{U_n}^{a_n}\,.
\end{equation}
\end{defn}

\begin{remark}
	Suppose $A = U \otimes U$ with $U$ atomic, then with $U_i = U$ for $1 \le i \le 4$
	\begin{equation*}
	\neg A = \neg( U_1 \otimes U_2 ) = \neg U_2 \parr \neg U_1
	\end{equation*}
	and \eqref{eq:intro_deduction_a} becomes
	\begin{center}
		\begin{tabular}{>{\centering}m{6cm} >{\centering}m{6cm} >{\centering}m{0.5cm}}
			$
			\begin{tikzcd}[column sep = small, row sep = small]
				& \ax\arrow[dl,bend right, dash]\arrow[dr,bend left, dash]\\
				\neg U_3 \parr \neg U_4 & & U_1 \otimes U_2
			\end{tikzcd}$
			&
			$
				\begin{tikzcd}[column sep = small, row sep = small]
					\neg U_3 \parr \neg U_4 \arrow[dr,bend right] & & U_1 \otimes U_2 \arrow[dl, bend left]\\
					& \cut
				\end{tikzcd}$
			&
			\tagarray{\label{eq:intro_deduction_a2}}
		\end{tabular}
	\end{center}
	The content of these links is $U_1 = U_4, U_2 = U_3$ (see \cite{algpnt}) and we represent these equations by link fermions $\psi_1, \psi_2$ respectively. The correspondence between fermion and equations can be represented informally as $\partial(\psi_1) = U_1 - U_4$, $\partial(\psi_2) = U_2 - U_3$. This relationship between fermions and equations can be made formal using Koszul matrix factorisations.
\end{remark}

\subsection{The Code}\label{section:the_code}

Let $\pi$ be a proof structure with Hilbert space $\call{H}_\pi$. The structure of $\pi$ lies in the fact that there is redundancy in the set of qubits: some unoriented atoms are represented \emph{twice} in the set of qubits $Q_\pi$. To be more precise, let $(U, y_U)$ be an oriented atom appearing in a formula $A$ on an edge in $\pi$ connecting two links $l \lto l'$ which are not conclusions, as in
\begin{equation}\label{eq:A_label_edge}
\xymatrix@C+2pc{
\cdots \quad l \ar[r]^-{A} & l' \quad \cdots
}
\end{equation}
The unoriented atom $U$ appears in both $U_l$ and $U_{l'}$ and there are consequently two qubits $\psi^l_U, \psi^{l'}_U$ in $\call{H}_\pi$ that are associated to $U$. In this section we introduce a self-adjoint operator $\Theta_U$ which represents the statement ``$U = U$'' for these two copies, and derive from these operators an error-correcting code. We write $\psi_U$ for $\psi^l_U$ and $\psi'_U$ for $\psi^{l'}_U$. Associated to these generators are operators on $\call{H}_\pi$ (see Section \ref{section:exterior_alg})
\[
	\psi_U = \psi_U \wedge (-)\,, \quad \psi_U^* = \psi_U^* \lrcorner (-)
\]
and similarly for $\psi'_U$.

\begin{defn}\label{def:edge_operators} The \emph{edge operator} on $\call{H}_\pi$ associated to $(U, y_U)$ is
	\begin{equation}
		\Theta^{l \lto l'}_{(U,y_U)} = y_U(\psi_U' - y_U \psi_U'^{\ast})(\psi_U + y_U \psi_U^\ast)\,.
	\end{equation}
\end{defn}

While the edge operator depends on the \emph{pair} consisting of an \emph{oriented} atom and the edge in $\pi$ on which it appears, to simplify the notation we often write $\Theta^{l \lto l'}_{U}$ or even just $\Theta_U$ where it will not cause confusion.

\begin{lemma} $\Theta_U$ is a self-adjoint operator on $\call{H}_\pi$.
\end{lemma}
\begin{proof}
Note that $\psi_U, \psi_U^*$ operate on a different tensor component to $\psi'_U, \psi_U'^*$ and so these two sets of operators anti-commute. Letting $(-)^\dagger$ the operator of taking adjoints,
\begin{align*}
	\Big[ y_U(\psi_U' - y_U \psi_U'^{\ast})(\psi_U + y_U \psi_U^\ast) \Big]^\dagger &= y_U (\psi_U + y_U \psi_U^\ast)^\dagger (\psi_U' - y_U \psi_U'^{\ast})^\dagger\\
	&= y_U (\psi_U^\dagger + y_U (\psi_U^\ast)^\dagger) ((\psi_U')^\dagger - y_U (\psi_U'^{\ast})^\dagger)\\
	&= y_U (\psi_U^* + y_U \psi_U) (\psi_U'^* - y_U \psi_U')\\
	&= - y_U  (\psi_U'^* - y_U \psi_U')(\psi_U^* + y_U \psi_U)\\
	&= ( -y_U \psi_U'^* + \psi_U')(\psi_U^* + y_U \psi_U)\\
	&= y_U ( -y_U \psi_U'^* + \psi_U')(y_U \psi_U^* + \psi_U)
\end{align*}
as claimed.
\end{proof}

Recall the isomorphism of Hilbert spaces $\Gamma$ from Definition \ref{defn:iso_gamma}. 

\begin{proposition}\label{prop:explicit_behaviour}
	Let $\pi$ be a proof structure with qubit ordering $U_1, \ldots, U_n$. As above let a particular oriented atom $(U, y_U)$ be chosen, let $U_i$ be the copy of $U$ in $U_l$ and $U_j$ the copy in $U_{l'}$. Then there is a commutative diagram
	\begin{equation}
	\begin{tikzcd}
		(\mathbb{C}^2)^{\otimes n}\arrow[r,"{\Gamma}"]\arrow[d,swap,"{F}"] & \call{H}_\pi\arrow[d,"{\Theta_U}"]\\
		(\mathbb{C}^2)^{\otimes n}\arrow[r,swap,"{\Gamma}"] & \call{H}_\pi
	\end{tikzcd}
\end{equation}
	where
	\begin{itemize}
		\item If $y_U = +$ and $j < i$ then $F = X_jZ_{j+1} \,\cdots\, Z_{i-1}X_i$. If $i = j + 1$ then $F = X_j X_i$.
		\item If $y_U = +$ and $j > i$ then $F = - X_i Z_i \,\cdots\, Z_j X_j$.
		\item If $y_U = -$ and $j > i$ then $F = X_i Z_{i+1} \,\cdots\, Z_{j-1}X_j$. If $i = j - 1$ then $F = X_i X_j$.
		\item If $y_U = -$ and $j < i$ then $F = X_j Z_j \,\cdots\, Z_i X_i$.
	\end{itemize}
\end{proposition}
\begin{proof}
	By definition $\Theta_U = y_U (\psi_{U_j} - y_U \psi_{U_j}^{\ast})(\psi_{U_i} + y_U \psi_{U_i}^\ast)$. Let $y_U = +$ and $j < i$. The identities \eqref{eq:iso_sneak_1}-\eqref{eq:iso_sneak_3} give $\psi_{U_i} \pm \psi_{U_i}^* = \pm Z_1 \,\cdots\, Z_{i-1} Z_i^{1 \mp 1} X_i$ and similarly for $U_j$. Hence
	\begin{align*}
		\Theta_U &= - Z_1 \,\cdots\, Z_j X_j Z_1 \,\cdots\, Z_{i-1} X_i\\
		&= X_j Z_1 \cdots Z_j Z_1 \cdots Z_{i-1} X_i\\
		&= X_j Z_{j+1} \cdots Z_{i-1} X_i
	\end{align*}
	as claimed. If $y_U = +$ and $j > i$ then
	\begin{align*}
		\Theta_U &= - Z_1 \,\cdots\, Z_j X_j Z_1 \,\cdots\, Z_{i-1} X_i\\
		&= Z_1 \cdots Z_j Z_1 \cdots Z_{i-1} X_i X_j\\
		&= Z_i \cdots Z_{j} X_i X_j\\
		&= - X_i Z_i \cdots Z_j X_j\,.
	\end{align*}
	If $y_U = -$ and $j > i$ then
	\begin{align*}
		\Theta_U &= Z_1 \cdots Z_{j-1} X_j Z_1 \cdots Z_i X_i\\
		&= Z_1 \cdots Z_{j-1} Z_1 \cdots Z_i X_i X_j\\
		&= Z_{i+1} \cdots Z_{j-1} X_i X_j\\
		&= X_i Z_{i+1} \cdots Z_{j-1} X_j\,.
	\end{align*}
	If $y_U = -$ and $j < i$ then
	\begin{align*}
		\Theta_U &= Z_1 \cdots Z_{j-1} X_j Z_1 \cdots Z_i X_i\\
		&= X_j Z_1 \cdots Z_{j-1} Z_1 \cdots Z_i X_i\\
		&= X_j Z_j \cdots Z_i X_i
	\end{align*}
	which completes the proof.
\end{proof}

\begin{corollary} For any oriented atom $(U,y_U)$ the operator $\Theta_U$ belongs to the Pauli group $G_n$ when viewed as an operator on $(\mathbb{C}^2)^{\otimes n}$ using any qubit ordering.
\end{corollary}

Recall that $\pi$ is an arbitrary proof structure.

\begin{defn}\label{defn:qecc_pi}
	The \emph{stabiliser quantum error-correcting code of $\pi$} is the pair
	\begin{equation}
		\den{\pi} = (\call{H}_\pi, S_\pi)
	\end{equation}
	where $S_\pi$ is the subgroup of the Pauli group generated by the operators $G_\pi = \{ \Theta_U \}_{U}$, with $U$ ranging over oriented atoms appearing in formulas decorating edges in $\pi$ connecting non-conclusion links. The \emph{codespace} of $\pi$ is the invariant subspace
	\begin{equation}
		\call{H}_{\pi}^{S_\pi} = \big\{ \varphi \in \call{H}_\pi \mid X \varphi = \varphi \text{ for all } X \in S_\pi \big\}\,.
		\end{equation}
	\end{defn}

\begin{defn} The \emph{Hamiltonian} of $\pi$ is the self-adjoint operator
	\begin{equation}\label{eq:hamiltonian_pi}
		H_\pi = - \sum_U \Theta_U
	\end{equation}
	where $U$ ranges over oriented atoms appearing in formulas decorating edges in $\pi$ connecting non-conclusion links.
\end{defn}

%Recall that each qubit is associated with a pair of unoriented atoms of $\pi$ (Lemma \ref{lemma:duality_links_edges}) and thus a \emph{pair} of qubits is \emph{a priori} associated with four unoriented atoms.

\begin{defn} Let $\pi$ be a proof structure. A qubit ordering $U_1, \ldots, U_n$ for $\pi$ is \emph{linear} if for every oriented atom $(U, y_U)$ of a formula $A$ decorating an edge $l \lto l'$ in $\pi$ connecting two non-conclusion links with corresponding atoms $U_i \in U_l, U_j \in U_{l'}$ we have
	\begin{itemize}
		\item If $y_U = +$ then $j = i+1$.
		\item If $y_U = -$ then $j = i-1$.
	\end{itemize}
\end{defn}

\begin{lemma}\label{lem:linearorderingexists}
	If $\pi$ is a proof net then there is a a linear qubit order associated to any ordering on the set of persistent paths of $\pi$.
\end{lemma}
\begin{proof}
Since $\pi$ is a proof net, by \cite[Proposition 4.11]{algpnt} we can write the set $U_\pi$ of unoriented atoms of $\pi$ \cite[Definition 3.16]{algpnt} as
\begin{align*}
&Z_1, \ldots, Z_{N_1},\\
&Z_{N_1+1}, \ldots, Z_{N_1+N_2},\\
& \qquad \qquad \cdots\\
&Z_{N_1 + \cdots + N_{c-1}+1},\ldots,Z_{N_1 + \cdots + N_{c-1} + N_c} = Z_{n}
\end{align*}
where $\mathscr{P}_i = (Z_{N_1 + \cdots + N_{i-1} + 1},\ldots,Z_{N_1 + \cdots + N_{i}})$ is the $i$th persistent path, of length $N_i$. As sets, persistent paths are the equivalence classes of the relation $\approx$ generated by $\sim$. The proof of the proposition shows that the only unordered pairs $V \sim V'$ are of the form $Z_t \sim Z_{t+1}$ for some consecutive pair in one of the above rows. By Lemma \ref{lemma:duality_links_edges} this means the qubits of $\pi$ are in bijection with consecutive pairs of unoriented atoms in the persistent paths of $\pi$.

If we write $Q_t$ for the qubit associated to $Z_t \sim Z_{t+1}$ (so $Q_{t+1}$ is associated to $Z_{t+1} \sim Z_{t+2}$ as long as $Z_{t+1}$ is not the final unoriented atom in its persistent path) then we claim
\begin{equation}\label{eq:candidate_linear}
Q_1, Q_2, \ldots, Q_{N_1 - 1}, Q_{N_1+1}, Q_{N_1+2}, \ldots, Q_{N_1 + \cdots + N_c - 1}
\end{equation}
is a linear qubit ordering for $\pi$ (the choice of how to globally order the persistent paths is arbitrary). This brings us to the role of the orientations. 

Let $\mathscr{P}$ be a persistent path and $\Bar{\mathscr{P}}$ the corresponding sequence of \emph{oriented} atoms. Roughly speaking, this path goes ``up from a conclusion to an axiom'' then performs a sequence (perhaps empty) of ``dips'' from axioms to cuts back to axioms, before going ``down from an axiom to a conclusion''. Informally the terms ``up'' and ``down'' refer to how we typically draw proof nets on the page, but formally these terms mean respectively \emph{against the edge orientation} and \emph{with the edge orientation}. More precisely, by \cite[Proposition 4.11]{algpnt} the sequence $\Bar{\mathscr{P}}$ begins with a negative atom $(Z_1,-)$ of the conclusion, followed by an \emph{initial segment} $(Z_1,-), (Z_2, y_2),\ldots,(Z_a,y_a)$ consisting of oriented atoms on edges leading up through $\otimes, \parr$ links (and no $(\ax)$,$(\cut)$ links) to the first oriented atom $(Z_a, y_a)$ occurring in a formula labelling an edge incident at an $(\ax)$ link. By inspection $y_2 = \cdots = y_a = -$. The next oriented atom in $\Bar{\mathscr{P}}$ is positive $(Z_{a+1}, +)$ and the atoms remain positive until the path hits a $(\cut)$ link, at which they switch to negative, and so on. The first segment (before the first $(\ax)$) consists entirely of negative atoms, the final segment (after the final $(\ax)$) of positive atoms.

From this we deduce that oriented atoms $(Z,y)$ in $\Bar{\mathscr{P}}$ occur with $y = -$ if the persistent path traverses the edge on which the atom occurs in the \emph{reverse} of its orientation in $\pi$, and occurs with $y = +$ if the persistent path traverses the edge in its normal direction. This is in fact the purpose of these orientations!

We are now prepared to prove \eqref{eq:candidate_linear} is a linear qubit ordering. Let $(U, y_U)$ be an oriented atom of a formula $A$ labelling an edge $l \lto l'$ between non-conclusion links, and let $\mathscr{P}$ be the persistent path containing $U$. Let $Q_i \in U_l, Q_j \in U_{l'}$ be the corresponding atoms. If $y_U = +$ then the persistent path follows the edge in its given orientation, so $j = i + 1$, and if $y_U = -$ then the persistent path follows the edge in the reverse orientation, so $i = j + 1$, as required.
\end{proof}

We have defined the Hilbert space $\call{H}_\pi$ and its Hamiltonian in a way that makes it clear that it depends only on \emph{local} information present at links in $\pi$. However, by paying attention to global structure we can give a simple presentation of the system involved. The relevant global structures in a proof net are the persistent paths.

\begin{thm}
	Let $\pi$ be a proof net with persistent paths $\mathscr{P}_1, \ldots, \mathscr{P}_c$ where the $i$th persistent path contains $N_i$ unoriented atoms. If $n$ is the number of qubits of $\pi$ then
	\begin{equation}
		(N_1 - 1) + \cdots + (N_c - 1) = n
	\end{equation}
	and under the isomorphism
	\begin{equation}
		\Gamma: (\mathbb{C}^2)^{\otimes n} \stackrel{\cong}{\lto} \call{H}_{\pi}
		\end{equation}
	corresponding to the associated linear qubit order, the stabiliser code of $\pi$ corresponds to the tensor product code
	\begin{align*}
		(\mathbb{C}^2)^{\otimes N_1 - 1} \otimes\qquad & \qquad X_1X_2, \ldots, X_{N_1 - 2}X_{N_1-1}\\
		(\mathbb{C}^2)^{\otimes N_2 - 1} \otimes\qquad &\qquad X_{N_1}X_{N_1 + 1}, \ldots, X_{N_1 + N_2-2}X_{N_1 + N_2 - 1},\\
		\qquad\vdots & \qquad\qquad\qquad\qquad\qquad\vdots\\
		(\mathbb{C}^2)^{\otimes N_c - 1}\qquad & \qquad X_{N_1 + \ldots + N_{c-1}}X_{N_1 + \ldots + N_{c-1}}, \ldots, X_{n-1}X_n
		\end{align*}
	\end{thm}
	\begin{proof}
	Follows from Proposition \ref{prop:explicit_behaviour} and Lemma \ref{lem:linearorderingexists}.
	\end{proof}

These tensor factors are well-known quantum error-correcting codes, whose associated systems are often referred to as ``quantum wires'' \cite{kitaev_wire}. In this language, the quantum system associated to a proof net $\pi$ is a tensor product of its persistent paths, realised as quantum wires. However this decomposition is \emph{global} structure, not immediately visible at the level of the edge operators $\Theta_U$ which define the \emph{local} structure.

\begin{remark}
	The error-correction procedure $P$ itself can be thought of as \emph{denoting the codespace}, since codewords are characterised by the condition $c = P(c)$. If we have two error-correcting codes $\call{H}' \subseteq \call{H}$ with the same codespace $\call{C}$, then the respectively error-correction procedures $P,P'$ have ``the same denotation but different senses''. The senses are different because in each case we allow a different set of errors. This simple observation is the whole content of the present paper - in this section we explain how we may associate to a proof $\pi$ a quantum error-correcting code $\call{H}, \call{C}, P$ and in the next section we show how we may associate to a reduction $\pi \rightsquigarrow \pi'$ an inclusion of codes $\call{H}' \subseteq \call{H}$.
\end{remark}

\section{Reductions}\label{section:reductions}

We have associated to any proof structure $\pi$ a Hilbert space $\call{H}_\pi$ and stabiliser quantum code $S_\pi$ with generating set $G_\pi$. We denote this code by $\den{\pi} = (\call{H}_\pi, S_\pi)$. In this section we associate to any reduction $\gamma: \pi \lto \pi'$ (Definition \ref{def:reduction}) a morphism of codes
\begin{equation}
	\den{\gamma}: \den{\pi} \lto \den{\pi'}\,.
\end{equation}

\subsection{A category of codes}

 Recall that a stabiliser quantum code $(\call{H}, S)$ consists of a finite-dimensional Hilbert space $\call{H}$ and stabiliser code $S$ on $\call{H}$ (see Definition \ref{defn:stabiliser_code}). A presentation of the Hilbert space as a tensor product of qubits is \emph{not} part of the data. Given a Hilbert space $\call{H}$ we write $U(\call{H})$ for the group of unitary operators. Given a subset $X$ of a group $G$ we denote by $\langle X \rangle$ the subgroup generated by $X$.

\begin{defn} A \emph{correction} $(\call{H}_1, S_1) \lto (\call{H}_2, S_2)$ of stabiliser codes is
	\begin{itemize}
		\item an inner-product preserving (hence injective) linear map $T: \call{H}_2 \lto \call{H}_1$
		\item subgroups $C, D \subseteq S_1$ with $\langle C, D \rangle = S_1$
		\item an isomorphism of groups $\nu: D \lto S_2$
	\end{itemize}
	satisfying the following conditions
	\begin{itemize}
		\item For every $x \in \call{H}_2$ the vector $T(x)$ is $C$-invariant
		\item The induced map $\Bar{T}: \call{H}_2 \lto \call{H}_1^{C}$ is an isomorphism of Hilbert spaces, such that
		\item The following diagram commutes
		\begin{equation}\label{eq:defn_correct}
		\xymatrix@C+3pc{
			U(\call{H}_2) \ar[r]_-{\cong}^-{\Bar{T} \circ (-) \circ \Bar{T}^{-1}} & U(\call{H}_1^C)\\
			S_2 \ar[u] & D \ar[l]^-{\nu}_-{\cong} \ar[u]
		}
		\end{equation}
		where the vertical maps are the canonical inclusion and restriction respectively.
	\end{itemize}
\end{defn}

\begin{remark}\label{remark:pres_invariant} Note that if $x \in \call{H}_1$ is $C$-invariant then $gx$ is $C$-invariant for any $g \in S_1$ since for $c \in C$, we have $cgx = gcx = gx$. Hence by restriction there is a well-defined morphism of groups $D \lto U(\call{H}_1^C)$ giving the right hand vertical map in \eqref{eq:defn_correct}.
\end{remark}

We are going to define a category of stabiliser codes and corrections. Given corrections $(T, C, D, \nu): (\call{H}_1, S_1) \lto (\call{H}_2,S_2)$ and $(T',C',D',\nu'): (\call{H}_2,S_2) \lto (\call{H}_3, S_3)$ we define the composite $(T'', C'', D'', \nu''): (\call{H}_1,S_1) \lto (\call{H}_3,S_3)$ as follows.

\begin{lemma} The data $(T'', C'', D'', \nu'')$ is a correction $(\call{H}_1,S_1) \lto (\call{H}_3,S_3)$ with
    \begin{itemize}
        \item $T'' = T \circ T'$,
        \item $C'' = \langle C, \nu^{-1}(C') \rangle$,
        \item $D'' = \nu^{-1}(D')$.
        \item $\nu''$ is the restriction of of $\nu$ to $D''$ followed by $\nu'$.
    \end{itemize}
\end{lemma}
\begin{proof}
It is clear that $T''$ is inner-product preserving and that $C'', D''$ are subgroups of $S_1$. We have
\begin{align*}
\langle C'', D'' \rangle &= \langle C, \nu^{-1}(C'), \nu^{-1}(D') \rangle\\
&= \langle C, \nu^{-1}(\langle C', D' \rangle) \rangle\\
&= \langle C, \nu^{-1}(S_2) \rangle\\
&= \langle C, D \rangle = S_1\,.    
\end{align*}
Given $x \in \call{H}_3$ we to prove that $T'' = TT'(x)$ is $C''$-invariant. But by hypothesis $T'(x)$ is $C'$-invariant, and $TT'(x)$ is $C$-invariant. If $g \in C''$ then $g T''(x) = h T''(x) = x$ since $h$ corresponds to an element of $C'$.

Next we have to prove that $T''$ induces an isomorphism $\call{H}_3 \cong \call{H}_1^{C''}$ but
\[
\call{H}_1^{C''} = (\call{H}_1^{C})^{C'} \cong \call{H}_2^{C'} \cong \call{H}_3\,.
\]
\end{proof}

\begin{defn} The \emph{category of quantum stabiliser codes} $\mathcal{SC}$ has as objects the stabiliser codes $(\call{H}, S)$ with corrections as morphisms.
\end{defn}

%But  Using \eqref{eq:defn_correct} we can 

%The category $\mathcal{SC}$ of quantum stabiliser codes has as objects pairs

Given a reduction $\pi \lto \pi'$ we will define a map $\gamma: \llbracket \pi \rrbracket \lto \llbracket \pi' \rrbracket$ depending on what type of reduction $\gamma$ is, see Section \ref{sec:proof_nets}. %First we define $\gamma$ in the context where $\pi \lto \pi'$ is an $a$-redex of the form \cite[Definition 3.0.2, (56)]{Troiani}, the case \cite[Definition 3.0.2, (57)]{Troiani} is similar.

\begin{defn} We label the relevant links of $\pi,\pi'$ according to the following diagram.
	% https://q.uiver.app/?q=WzAsMTAsWzQsMCwiXFxzdGFja3JlbHtsfXtcXGJ1bGx9Il0sWzQsMSwiQSJdLFszLDIsIlxcc3RhY2tyZWx7bF97XFxjdXR9fXtcXGN1dH0iXSxbMiwxLCJcXG5lZyBBIl0sWzEsMCwiXFxzdGFja3JlbHtsX1xcYXh9e1xcYXh9Il0sWzAsMSwiQSJdLFswLDIsIlxcdmRvdHMiXSxbNSwwLCJcXHN0YWNrcmVse2x9e1xcYnVsbH0iXSxbNSwxLCJBIl0sWzUsMiwiXFx2ZG90cyJdLFs1LDZdLFs0LDUsIiIsMCx7ImN1cnZlIjoyLCJzdHlsZSI6eyJoZWFkIjp7Im5hbWUiOiJub25lIn19fV0sWzQsMywiIiwyLHsiY3VydmUiOi0yLCJzdHlsZSI6eyJoZWFkIjp7Im5hbWUiOiJub25lIn19fV0sWzMsMiwiIiwyLHsiY3VydmUiOjJ9XSxbMSwyLCIiLDAseyJjdXJ2ZSI6LTJ9XSxbMCwxLCIiLDAseyJzdHlsZSI6eyJoZWFkIjp7Im5hbWUiOiJub25lIn19fV0sWzcsOCwiIiwwLHsic3R5bGUiOnsiaGVhZCI6eyJuYW1lIjoibm9uZSJ9fX1dLFs4LDldXQ==
	\begin{equation}\label{eq:a_redex_labelling}
	\begin{tikzcd}[column sep = small, row sep = small]
		& {\stackrel{l_{\ax}}{\ax}} &&& {\stackrel{l}{\bullet}} & {\stackrel{l}{\bullet}} \\
		A && {\neg A} && A & A \\
		\stackrel{m}{\bullet} &&& {\stackrel{l_{\cut}}{\cut}} && \stackrel{m}{\bullet}
		\arrow[from=2-1, to=3-1]
		\arrow[curve={height=12pt}, no head, from=1-2, to=2-1]
		\arrow[curve={height=-12pt}, no head, from=1-2, to=2-3]
		\arrow[curve={height=12pt}, from=2-3, to=3-4]
		\arrow[curve={height=-12pt}, from=2-5, to=3-4]
		\arrow[no head, from=1-5, to=2-5]
		\arrow[no head, from=1-6, to=2-6]
		\arrow[from=2-6, to=3-6]
	\end{tikzcd}
\end{equation}
	For each oriented atom $(U,y)$ of $A$ we define a $\bb{Z}_2$-degree zero map for $y = +$ by:
	\begin{align}
		\gamma_U: \bigwedge \bb{C} \psi_U^l &\lto \bigwedge \bb{C} \psi_U^l \otimes \bigwedge \bb{C} \psi_U^{l_{\cut}} \otimes \bigwedge \bb{C} \psi_U^{l_{\ax}}\label{eq:oriented_ax_cod}\\
		\ket{j} &\longmapsto \frac{1}{\sqrt{2}}(\ket{+++} + (-1)^j\ket{---})
	\end{align}
	If $y = -$ then $\gamma_U$ has the same domain and formula, but its codomain is:
	\begin{equation}
		\bigwedge \bb{C}\psi_{U}^{l_{\ax}} \otimes \bigwedge \bb{C}\psi_{U}^{l_{\cut}} \otimes \bigwedge \bb{C}\psi_U^{l}
	\end{equation}
	If $m \neq l$ is a link of $\pi'$ and $V$ an unoriented atom of $m$, then $m$ is in $\pi$ and we define $\gamma_V: \bigwedge \bb{C}\psi_V^m \lto \bigwedge \bb{C}\psi_V^m$ to be the identity.
	
	Assume now that we have a linear order $U_1 < \hdots < U_r$ of $\pi$. Then in all cases of $U,V$ above, post composing with an inclusion induces a morphism with codomain:
	\begin{equation}
		\bigwedge\bb{C}\psi_{U_1} \otimes \hdots \otimes \bigwedge\bb{C}\psi_{U_r}
		\end{equation}
	Assuming now that $V_1 < \hdots < V_{r'}$ is a linear order of $\pi'$, we tensor over all morphisms considered to thus obtain a morphism:
	\begin{equation}
		\bigwedge\bb{C}\psi_{V_1} \otimes \hdots \otimes \bigwedge\bb{C}\psi_{V_{r'}} \lto \bigwedge \bb{C} \psi_{U_1} \otimes \hdots \otimes \bigwedge \bb{C}\psi_{U_r}
		\end{equation}
	Finally, pre and post composing with the appropriate isomorphisms we obtain the morphism of interest:
	\begin{equation}
		\gamma: \call{H}_{\pi'} \lto \call{H}_\pi
	\end{equation}

Now we define the subset $C_{\pi} \subseteq S_\pi$. Let $\call{A}$ be the unoriented atoms of $\neg A$ and hence also of $A$.
\begin{equation}
	C_{\pi} = \lbrace \Theta^{l_{\ax} \lto l_{\cut}}_U\rbrace_{U \in \call{A}} \cup \lbrace \Theta^{l \lto l_{\cut}}_U\rbrace_{U \in \call{A}}
	\end{equation}
Next, we define $\gamma$ in the case when $\pi \lto \pi'$ reduces an $m$-redex. For convenience, we label the links involved in the reduction according to the following Diagram (note: there may be some equalities among $m_1,m_2,m_3,m_4$). As in the $a$-redex case, the labels $m_1,m_2,m_3,m_4$ will be used in the proof of Lemma \ref{lem:isomorphism} but not in the current Definition.
\begin{equation}\label{eq:m_redex_labelling}
\begin{tikzcd}[column sep = small, row sep = small]
	{} && {\stackrel{m_1}{\bullet}} && {\stackrel{m_2}{\bullet}} && {\stackrel{m_3}{\bullet}} && {\stackrel{m_4}{\bullet}} \\
	&& A && B && {\neg A} && {\neg B} \\
	&&& \otimes &&&& \parr \\
	&&& {l_{\otimes}\quad A \otimes B} &&&& {\neg A \parr \neg B\quad l_{\parr}} \\
	&&&&& {\stackrel{m}{\cut}} \\
	&& {\stackrel{m_1}{\bullet}} && {\stackrel{m_2}{\bullet}} && {\stackrel{m_3}{\bullet}} && {\stackrel{m_4}{\bullet}} \\
	&& A && B && {\neg A} && {\neg B} \\
	&&&& {\stackrel{a}{\cut}} &&& {\stackrel{b}{\cut}}
	\arrow[curve={height=12pt}, from=2-3, to=3-4]
	\arrow[curve={height=-12pt}, from=2-5, to=3-4]
	\arrow[curve={height=12pt}, from=2-7, to=3-8]
	\arrow[curve={height=-12pt}, from=2-9, to=3-8]
	\arrow[no head, from=1-7, to=2-7]
	\arrow[no head, from=1-9, to=2-9]
	\arrow[no head, from=3-8, to=4-8]
	\arrow[curve={height=-12pt}, from=4-8, to=5-6]
	\arrow[no head, from=3-4, to=4-4]
	\arrow[no head, from=1-3, to=2-3]
	\arrow[no head, from=1-5, to=2-5]
	\arrow[curve={height=-12pt}, from=7-7, to=8-5]
	\arrow[curve={height=12pt}, from=7-3, to=8-5]
	\arrow[curve={height=12pt}, from=7-5, to=8-8]
	\arrow[curve={height=-12pt}, from=7-9, to=8-8]
	\arrow[no head, from=6-9, to=7-9]
	\arrow[no head, from=6-7, to=7-7]
	\arrow[no head, from=6-5, to=7-5]
	\arrow[no head, from=6-3, to=7-3]
	\arrow[curve={height=12pt}, from=4-4, to=5-6]
\end{tikzcd}
\end{equation}
For each oriented atom $(U,y_u)$ of $A$ and $(V,y_v)$ of $B$ we define $\bb{Z}_2$-degree zero maps:
\begin{align}
	\gamma_U: \bigwedge \bb{C} \psi_U^a \lto \bigwedge \bb{C}\psi_U^\parr \otimes \bigwedge \bb{C}\psi_U^{\cut} \otimes \bigwedge \bb{C}\psi_U^{\otimes},\qquad y_u = +\\
	\gamma_U: \bigwedge \bb{C} \psi_U^a \lto \bigwedge \bb{C}\psi_U^\otimes \otimes \bigwedge \bb{C}\psi_U^{\cut} \otimes \bigwedge \bb{C}\psi_U^{\parr},\qquad y_u = -\\
	\gamma_V: \bigwedge \bb{C} \psi_V^b \lto \bigwedge \bb{C}\psi_V^\parr \otimes \bigwedge \bb{C}\psi_V^{\cut} \otimes \bigwedge \bb{C}\psi_V^{\otimes},\qquad y_v = +\\
	\gamma_V: \bigwedge \bb{C} \psi_V^b \lto \bigwedge \bb{C}\psi_V^{\otimes} \otimes \bigwedge \bb{C}\psi_V^{\cut} \otimes \bigwedge \bb{C}\psi_V^{\parr},\qquad y_v = -
	\end{align}
all by the following formula.
\begin{equation}
	\ket{j} \longmapsto \frac{1}{\sqrt{2}}(\ket{+++} + (-1)^j\ket{---})
\end{equation}
Following the same procedure as in the case when the reduction $\pi \lto \pi'$ reduced an $a$-redex, we tensor over all links with respect to the order given by the linear order on $\pi$ and then compose with the relevant isomorphisms to obtain the following, injective, $\bb{Z}_2$-degree zero map
\begin{equation}
	\gamma: \call{H}_{\pi'} \lto \call{H}_\pi
	\end{equation}
Finally, we define $C_{\pi}$ in this case. Let $\call{A}$ denote the unoriented atoms of $A$ (and hence of $\neg A$) and $\call{B}$ that of $B$ (and hence of $\neg B$).
\begin{equation}
	C_{\pi} = \lbrace \Theta^{l_{\otimes} \lto l_{\cut}}_U, \Theta^{l_{\otimes} \lto l_{\cut}}_V\rbrace_{U\in \call{A}, V \in \call{B}} \cup \lbrace \Theta^{l_{\parr} \lto l_{\cut}}_U, \Theta^{l_{\parr} \lto l_{\cut}}_V \rbrace_{U \in \call{A}, V \in \call{B}}
	\end{equation}
\end{defn}
The next Lemma states that the map $\gamma$ factors through an isomorphism.

\begin{lemma}\label{lem:isomorphism}
	There exists an isomorphism $\call{H}_{\pi'} \lto \call{H}_{\pi}^{C_{\pi}}$ such that the following diagram commutes, where $\operatorname{inc}: \call{H}_{\pi}^{C_\pi} \lto \call{H}_{\pi}$ is an inclusion.
	\begin{equation}
		\begin{tikzcd}
	\call{H}_{\pi'}\arrow[dr,"{\cong}"]\arrow[rr,"{\gamma}"] && \call{H}_{\pi}\\
	& \call{H}^{C_{\pi}}\arrow[ur,"{\operatorname{inc}}"]
			\end{tikzcd}
		\end{equation}
	\end{lemma}
\begin{proof}
	There are two similar cases, when $\pi \lto \pi'$ reduces an $a$-redex, and when $\pi \lto \pi'$ reduces an $m$-redex. We prove the case when $\pi \lto \pi'$ reduces an $a$-redex and omit the remaining details.
	
	If $A$ has $n$ unoriented atoms then $\operatorname{dim}_{\bb{C}}\call{H}_\pi = (2^{n})^2\operatorname{dim}_{\bb{C}}\call{H}_{\pi'}$ and $|C_\pi| = 2n$, so by a dimension argument it suffices to show $\operatorname{Im}(\gamma) \subseteq \call{H}_\pi^{C_\pi}$.
	
	This can be done one unoriented atom $U$ at a time, so it reduces to showing that $\operatorname{Im}(\gamma_U)$ is contained in the part of the right hand side of $\eqref{eq:oriented_ax_cod}$ invariant under $\Theta_U^{\ax}, \Theta_U^l$. Let $y$ be such that $(U,y)$ is an oriented atom of $A$. Then if $y = +$ by Proposition \ref{prop:explicit_behaviour}, $\Theta^{\ax}_U, \Theta_U^l$ act respectively as $X_2X_3$ and $X_1X_2$ on the following space.
	\begin{equation}
		\bigwedge \bb{C} \psi_U^l \otimes \bigwedge \bb{C}\psi_U^{\cut} \otimes \bigwedge \bb{C}\psi_U^{\ax}
	\end{equation}
	If $y = -$ then $\Theta_U^{\ax}, \Theta_U^l$ act as (respectively) $X_1X_2, X_2X_3$ on the following space.
	\begin{equation}
		\bigwedge \bb{C}\psi_U^{\ax} \bigwedge \bb{C}\psi_U^{\cut} \otimes\bigwedge \bb{C} \psi_U^l
	\end{equation}
	So it suffices to show $\frac{1}{\sqrt{2}}(\ket{+++} + (-1)^j\ket{---})$ is invariant under $\lbrace X_1X_2, X_2X_3\rbrace$ which is clear.
\end{proof}

\begin{lemma}\label{lem:commuting_square}
	For each $g \in S_\pi\setminus C_\pi$ there is a unique $g' \in S_{\pi'}$ such that the diagram below commutes:
	\begin{equation}
		\begin{tikzcd}\label{eq:bijective_square}
			\call{H}_{\pi'}\arrow[r,"{\gamma}"]\arrow[d,"{g'}"] & \call{H}_{\pi}\arrow[d,"{g}"]\\
			\call{H}_{\pi'}\arrow[r,swap,"{\gamma}"] & \call{H}_{\pi}
			\end{tikzcd}
		\end{equation}
	and this map $g \longmapsto g'$ is a bijection $S_{\pi}\setminus C_{\pi} \lto S_{\pi'}$.
	\end{lemma}
\begin{proof}
	First we consider the case when $\pi \lto \pi'$ reduces an $a$-redex.

	Any $g \in S_{\pi}\setminus C_\pi$ is an edge operator $\Theta_U$ associated to an oriented atom $U$ on an edge in $\pi$ not equal to the edges $l_{\ax} \lto l_{\cut}, l \lto l_{\cut}$ shown in \eqref{eq:a_redex_labelling}. If the edge does not begin at $l_{\ax}$ then the claim is easily verified, as the same edge exists in $\pi'$. Suppose the edge begins at $l_{\ax}$ so $(U,y_u)$ is an oriented atom of $\neg A$ and the edge is $l_{\ax} \lto m$ with $m$ not a conclusion. We take $g'$ to be $\Theta_U^{l\lto m}$, which is an edge operator associated to an edge in $\pi'$.
	
	If $y_u = +$ then $U$ is negatively oriented in $\neg A$ and by Proposition \ref{prop:explicit_behaviour}, commutativity of \eqref{eq:bijective_square} follows from commutativity of the following Diagram. In what follows, the morphism $\operatorname{id}$ is the identity on the space $\bigwedge \bb{C}\psi_U^m$.
	\begin{equation}
		\begin{tikzcd}[column sep = huge]
			\bigwedge \bb{C} \psi_U^m \otimes \bigwedge \bb{C}\psi_U^l \arrow[r,"{\operatorname{id} \otimes \gamma_U}"]\arrow[d,swap,"{X_1X_2}"]& \bigwedge \bb{C}\psi_U^m \otimes \bigwedge \bb{C}\psi_U^{l_{\ax}} \otimes \bigwedge \bb{C}\psi_U^{l_{\cut}} \otimes \bigwedge \bb{C}\psi_U^{l}\arrow[d,"{X_1X_2}"]\\
			\bigwedge \bb{C}\psi_U^m \otimes \bigwedge \bb{C} \psi_U^l \arrow[r, "{\operatorname{id} \otimes \gamma_U}"] & \bigwedge \bb{C} \psi_U^m \otimes \bigwedge \bb{C}\psi_U^{l_{\ax}} \otimes \bigwedge \bb{C}\psi_U^{l_{\cut}} \otimes \bigwedge \bb{C}\psi_U^l\bigwedge
			\end{tikzcd}
		\end{equation}
	This can be checked by the following calculation, in what follows the notation $\overline{i}$ denotes a bitflip, that is, $\overline{i} =0$ if $i = 1$ and $\overline{i} = 1$ if $i = 0$, similarly for $j$.
	\begin{align*}
		X_1X_2(\operatorname{id} \otimes \gamma_U)\ket{ij} &= \frac{1}{\sqrt{2}}X_1X_2(\ket{i+++} + (-1)^j\ket{i---})\\
		&= \frac{1}{\sqrt{2}}(\ket{\overline{i}+++} + (-1)^{j+1}\ket{\overline{i}---})\\
		&= (\operatorname{id} \otimes \gamma) \ket{\overline{i} \overline{j}}\\
		&= (\operatorname{id} \otimes \gamma) X_1 X_2 \ket{ij}
		\end{align*}
	
If $y_u = -$ then the argument is similar; $U$ is positively oriented in $\neg A$ and commutativity of \eqref{eq:bijective_square} follows by Proposition \ref{prop:explicit_behaviour} from commutativity of the following Diagram. In what follows, the morphism $\operatorname{id}$ is the identity on the space $\bigwedge \bb{C}\psi_U^m$.
\begin{equation}
	\begin{tikzcd}[column sep = huge]
		\bigwedge \bb{C}\psi_U^m \otimes \bigwedge \bb{C} \psi_U^l\arrow[r,"{\operatorname{id} \otimes \gamma_U}"]\arrow[d,swap,"{X_1X_2}"] & \bigwedge \bb{C}\psi_U^l \otimes \bigwedge\bb{C}\psi_U^{\cut} \otimes \bigwedge \bb{C}\psi_U^{\ax} \otimes \bigwedge \bb{C}\psi_U^m\arrow[d, "{X_1X_2}"]\\
		\bigwedge \bb{C} \psi_U^m \otimes \bigwedge \bb{C} \psi_U^l\arrow[r,swap,"{\operatorname{id} \otimes \gamma_U}"] & \bigwedge \bb{C}\psi_U^l \otimes \bigwedge \bb{C}\psi_U^{\cut} \otimes \bigwedge \bb{C} \psi_U^{\ax} \otimes \bigwedge \bb{C} \psi_U^m
		\end{tikzcd}
	\end{equation}
To see this, we observe the following calculation, in what follows the notation $\overline{i}$ denotes a bitflip, that is, $\overline{i} =0$ if $i = 1$ and $\overline{i} = 1$ if $i = 0$, similarly for $j$.
\begin{align*}
	X_1X_2(\operatorname{id} \otimes \gamma_U)\ket{ij} &= \frac{1}{\sqrt{2}}X_1X_2(\ket{i+++} + (-1)^j\ket{i---})\\
	&= \frac{1}{\sqrt{2}}(\ket{\overline{i}+++} + (-1)^{j+1}\ket{\overline{i}---})\\
	&= (\operatorname{id} \otimes \gamma) \ket{\overline{i}\overline{j}}\\
	&= (\operatorname{id} \otimes \gamma_U) X_1 X_2 \ket{ij}
	\end{align*}
The claim about a bijection follows immediately from the fact that $\gamma_U$ is injective.

Now we consider the case where $\pi \lto \pi'$ reduces an $m$-redex, we refer to \eqref{eq:m_redex_labelling}.

We show the details for the edge operator $\Theta_U^{m_1 \lto l_{\otimes}}$. We take $g'$ to be $\Theta_U^{m_1 \lto a}$, which is an edge operator corresponding to an edge in $\pi'$. We check the $y_u = +$ case. It suffices by Proposition \ref{prop:explicit_behaviour} to show that the following diagram commutes
\begin{equation}
	\begin{tikzcd}[column sep = huge]
		\bigwedge \bb{C} \psi_U^a \otimes \bigwedge \bb{C}\psi_U^{m_1}\arrow[r,"{\gamma_U \otimes \operatorname{id}}"]\arrow[d,swap,"{X_1X_2}"] & \bigwedge \bb{C}\psi_U^{l_{\parr}} \otimes \bigwedge \bb{C} \psi_U^{l_{\cut}} \otimes \bigwedge \bb{C}\psi_U^{l_{\otimes}} \otimes \bigwedge \bb{C} \psi_U^{m_1}\arrow[d,"{X_3X_4}"]\\
		\bigwedge \bb{C}\psi_U^{a} \otimes \bigwedge \bb{C}\psi_U^{m_1}\arrow[r,"{\gamma_U \otimes \operatorname{id}}"] & \bigwedge \bb{C}\psi_U^{l_{\parr}} \otimes \bigwedge \bb{C}\psi_U^{l_{\cut}} \otimes \bigwedge \bb{C}\psi_U^{l_{\otimes}} \otimes \bigwedge \bb{C}\psi_U^{m_1}
		\end{tikzcd}
	\end{equation}
Again, this follows from a calculation.
\begin{align*}
X_3X_4 (\gamma_U \otimes \operatorname{id})\ket{ji} &= \frac{1}{\sqrt{2}}X_3X_4(\ket{+++i} + (-1)^j\ket{---i})\\
&= \frac{1}{\sqrt{2}}(\ket{+++\overline{i}} + (-1)^{\overline{j}\ket{---\overline{i}}})\\
&= (\gamma_U \otimes \operatorname{id})\ket{\overline{j}\overline{i}}\\
&= (\gamma_U \otimes \operatorname{id})X_1X_2\ket{ji}
\end{align*}
	\end{proof}
Propositio \ref{prop:explicit_behaviour} and Lemma \ref{lem:commuting_square} fit together into the following Theorem.
\begin{thm}[The Reduction Theorem]\label{thm:cut_model}
	For each reduction $\gamma: \pi \lto \pi'$ there exists a subset $C_\pi \subseteq S_\pi$ and an isomorphism:
	\begin{equation}
		\hat{\gamma}: \call{H}_{\pi'} \lto \call{H}_\pi^{C_\pi}
	\end{equation}
	such that for every $g \in S_\pi \setminus C_\pi$ there is a unique $g' \in S_{\pi'}$ making the following diagram commute:
	\begin{equation}\label{eq:cut_model_square}
		\begin{tikzcd}
			\call{H}_{\pi'}\arrow[r,"{\hat{\gamma}}"]\arrow[d,"{g'}"] & \call{H}_{\pi}^{C_\pi}\arrow[d,swap,"{g}"]\\
			\call{H}_{\pi'}\arrow[r,"\hat{\gamma}"] & \call{H}_{\pi}^{C_\pi}
		\end{tikzcd}
	\end{equation}
	and this map $g \longmapsto g'$ is a bijection $S_\pi \setminus C_\pi \lto S_{\pi'}$.
\end{thm}
\begin{proof}
	By Lemma \ref{lem:isomorphism} and Lemma \ref{lem:commuting_square}.
	\end{proof}

%\begin{remark} What if one reduction is part of another one? Or there is a commutative diagram?
%\end{remark}

\section{Examples}\label{section:links_as_operators}

\begin{example}\label{ex:CatGoI_ex}
	
	Consider the following proof net, which we denote by $\pi$. The labels on the variables are artificial, so that $U_i$ denotes the (atomic) variable $U$ for all $1 \le i \le 10$, and each link is labelled.
	% https://q.uiver.app/?q=WzAsMTYsWzEsMCwiXFxzdGFja3JlbHtsX3syfX17XFxheH0iXSxbMCwxLCJcXG5lZyBYXzEiXSxbMiwxLCJYXzIiXSxbMCwyLCJcXHN0YWNrcmVse2xfMX17XFxvcGVyYXRvcm5hbWV7Y319Il0sWzMsMiwiXFxzdGFja3JlbHtsXzN9e1xcb3RpbWVzfSJdLFs0LDEsIlxcbmVnIFhfMyJdLFs2LDEsIlhfNCJdLFs1LDAsIlxcc3RhY2tyZWx7bF80fXtcXGF4fSJdLFs2LDIsIlxcc3RhY2tyZWx7bF81fXtcXG9wZXJhdG9ybmFtZXtjfX0iXSxbMywzLCJYXzcgXFxvdGltZXMgXFxuZWcgWF84Il0sWzgsMywiXFxuZWcgWF85IFxccGFyciBYX3sxMH0iXSxbOCwyLCJcXHN0YWNrcmVse2xfN317XFxwYXJyfSJdLFs3LDEsIlxcbmVnIFhfNSJdLFs5LDEsIlhfNiJdLFs4LDAsIlxcc3RhY2tyZWx7bF82fXtcXGF4fSJdLFs2LDQsIlxcc3RhY2tyZWx7bF84fXtcXGN1dH0iXSxbMCwxLCIiLDAseyJjdXJ2ZSI6Miwic3R5bGUiOnsiaGVhZCI6eyJuYW1lIjoibm9uZSJ9fX1dLFswLDIsIiIsMix7ImN1cnZlIjotMiwic3R5bGUiOnsiaGVhZCI6eyJuYW1lIjoibm9uZSJ9fX1dLFsxLDNdLFs3LDUsIiIsMCx7ImN1cnZlIjoyLCJzdHlsZSI6eyJoZWFkIjp7Im5hbWUiOiJub25lIn19fV0sWzcsNiwiIiwyLHsiY3VydmUiOi0yLCJzdHlsZSI6eyJoZWFkIjp7Im5hbWUiOiJub25lIn19fV0sWzYsOF0sWzUsNCwiIiwwLHsiY3VydmUiOi0yfV0sWzIsNCwiIiwyLHsiY3VydmUiOjJ9XSxbMTQsMTIsIiIsMix7ImN1cnZlIjoyLCJzdHlsZSI6eyJoZWFkIjp7Im5hbWUiOiJub25lIn19fV0sWzE0LDEzLCIiLDAseyJjdXJ2ZSI6LTIsInN0eWxlIjp7ImhlYWQiOnsibmFtZSI6Im5vbmUifX19XSxbMTMsMTEsIiIsMCx7ImN1cnZlIjotMn1dLFsxMiwxMSwiIiwyLHsiY3VydmUiOjJ9XSxbMTEsMTAsIiIsMix7InN0eWxlIjp7ImhlYWQiOnsibmFtZSI6Im5vbmUifX19XSxbMTAsMTUsIiIsMix7ImN1cnZlIjotMn1dLFs5LDE1LCIiLDAseyJjdXJ2ZSI6Mn1dLFs0LDksIiIsMCx7InN0eWxlIjp7ImhlYWQiOnsibmFtZSI6Im5vbmUifX19XV0=
	\[\begin{tikzcd}[column sep = small, row sep = small]
		& {\stackrel{l_{2}}{\ax}} &&&& {\stackrel{l_4}{\ax}} &&& {\stackrel{l_6}{\ax}} \\
		{\neg U_1} && {U_2} && {\neg U_3} && {U_4} & {\neg U_5} && {U_6} \\
		{\stackrel{l_1}{\operatorname{c}}} &&& {\stackrel{l_3}{\otimes}} &&& {\stackrel{l_5}{\operatorname{c}}} && {\stackrel{l_7}{\parr}} \\
		&&& {U_7 \otimes \neg U_8} &&&&& {\neg U_9 \parr U_{10}} \\
		&&&&&& {\stackrel{l_8}{\cut}}
		\arrow[curve={height=12pt}, no head, from=1-2, to=2-1]
		\arrow[curve={height=-12pt}, no head, from=1-2, to=2-3]
		\arrow[from=2-1, to=3-1]
		\arrow[curve={height=12pt}, no head, from=1-6, to=2-5]
		\arrow[curve={height=-12pt}, no head, from=1-6, to=2-7]
		\arrow[from=2-7, to=3-7]
		\arrow[curve={height=-12pt}, from=2-5, to=3-4]
		\arrow[curve={height=12pt}, from=2-3, to=3-4]
		\arrow[curve={height=12pt}, no head, from=1-9, to=2-8]
		\arrow[curve={height=-12pt}, no head, from=1-9, to=2-10]
		\arrow[curve={height=-12pt}, from=2-10, to=3-9]
		\arrow[curve={height=12pt}, from=2-8, to=3-9]
		\arrow[no head, from=3-9, to=4-9]
		\arrow[curve={height=-12pt}, from=4-9, to=5-7]
		\arrow[curve={height=12pt}, from=4-4, to=5-7]
		\arrow[no head, from=3-4, to=4-4]
	\end{tikzcd}\]
We now replace each link of $\pi$ with its set of unoriented atoms and erase the labels on the edges and replace them with their associated operators in the code:
% https://q.uiver.app/?q=WzAsMTQsWzEsMCwiXFx7VV8xXFx9ICJdLFsyLDEsIlxce1hfMV57bF8yfVhfM157bF8zfVxcfSJdLFswLDEsIlxcdmFybm90aGluZyJdLFszLDIsIlxce1VfMyxVXzRcXH0iXSxbNCwxLCJcXHtYXzRee2xfM31YXzJee2xfNH1cXH0iXSxbNSwwLCJcXHtVXzJcXH0iXSxbNiwxLCJcXHZhcm5vdGhpbmciXSxbMywzLCJcXHtYXzNee2xfM31YXzhee2xfM30sIFhfNF57bF8zfVhfOV57bF84fVxcfSJdLFs4LDMsIlxce1hfNl57bF83fVhfOF57bF8zfSwgWF83XntsXzd9WF85XntsXzh9XFx9Il0sWzgsMiwiXFx7VV82LCBVXzdcXH0iXSxbNywxLCJcXHtYXzVee2xfNn1YXzZee2xfN31cXH0iXSxbOSwxLCJcXHtYXzVee2xfNn1YXzdee2xfN31cXH0iXSxbOCwwLCJcXHtVXzVcXH0iXSxbNiw0LCJcXHtVXzgsIFVfOVxcfSJdLFswLDEsIiIsMix7ImN1cnZlIjotMiwic3R5bGUiOnsiaGVhZCI6eyJuYW1lIjoibm9uZSJ9fX1dLFs1LDQsIiIsMCx7ImN1cnZlIjoyLCJzdHlsZSI6eyJoZWFkIjp7Im5hbWUiOiJub25lIn19fV0sWzQsMywiIiwwLHsiY3VydmUiOi0yfV0sWzEsMywiIiwyLHsiY3VydmUiOjJ9XSxbMTIsMTAsIiIsMix7ImN1cnZlIjoyLCJzdHlsZSI6eyJoZWFkIjp7Im5hbWUiOiJub25lIn19fV0sWzEyLDExLCIiLDAseyJjdXJ2ZSI6LTIsInN0eWxlIjp7ImhlYWQiOnsibmFtZSI6Im5vbmUifX19XSxbMTEsOSwiIiwwLHsiY3VydmUiOi0yfV0sWzEwLDksIiIsMix7ImN1cnZlIjoyfV0sWzksOCwiIiwyLHsic3R5bGUiOnsiaGVhZCI6eyJuYW1lIjoibm9uZSJ9fX1dLFs4LDEzLCIiLDIseyJjdXJ2ZSI6LTJ9XSxbNywxMywiIiwwLHsiY3VydmUiOjJ9XSxbMyw3LCIiLDAseyJzdHlsZSI6eyJoZWFkIjp7Im5hbWUiOiJub25lIn19fV0sWzAsMiwiIiwwLHsiY3VydmUiOjJ9XSxbNSw2LCIiLDIseyJjdXJ2ZSI6LTJ9XV0=
\[\adjustbox{scale=0.70}{\begin{tikzcd}[column sep = tiny]
	& {\{U_1\} } &&&& {\{U_2\}} &&& {\{U_5\}} \\
	\varnothing && {\{X_1^{l_2}X_3^{l_3}\}} && {\{X_4^{l_3}X_2^{l_4}\}} && \varnothing & {\{X_5^{l_6}X_6^{l_7}\}} && {\{X_5^{l_6}X_7^{l_7}\}} \\
	&&& {\{U_3,U_4\}} &&&&& {\{U_6, U_7\}} \\
	&&& {\{X_3^{l_3}X_8^{l_3}, X_4^{l_3}X_9^{l_8}\}} &&&&& {\{X_6^{l_7}X_8^{l_3}, X_7^{l_7}X_9^{l_8}\}} \\
	&&&&&& {\{U_8, U_9\}}
	\arrow[curve={height=-12pt}, no head, from=1-2, to=2-3]
	\arrow[curve={height=12pt}, no head, from=1-6, to=2-5]
	\arrow[curve={height=-12pt}, from=2-5, to=3-4]
	\arrow[curve={height=12pt}, from=2-3, to=3-4]
	\arrow[curve={height=12pt}, no head, from=1-9, to=2-8]
	\arrow[curve={height=-12pt}, no head, from=1-9, to=2-10]
	\arrow[curve={height=-12pt}, from=2-10, to=3-9]
	\arrow[curve={height=12pt}, from=2-8, to=3-9]
	\arrow[no head, from=3-9, to=4-9]
	\arrow[curve={height=-12pt}, from=4-9, to=5-7]
	\arrow[curve={height=12pt}, from=4-4, to=5-7]
	\arrow[no head, from=3-4, to=4-4]
	\arrow[curve={height=12pt}, from=1-2, to=2-1]
	\arrow[curve={height=-12pt}, from=1-6, to=2-7]
\end{tikzcd}}\]
Associated to each of these links is a Hilbert space, in the following we write $\psi_i$ for $\psi_{U_i}$ for $i = 1,...,9$.
\begin{align*}
	\call{H}_{l_1} &= \bb{C} & \call{H}_{l_2} &= \bigwedge \bb{C}\psi_{1} & \call{H}_{l_3} &= \bigwedge (\bb{C}\psi_{3} \oplus \bb{C}\psi_4)\\
	\call{H}_{l_4} &= \bigwedge \bb{C}\psi_{2} & \call{H}_{l_5} &= \bb{C} & \call{H}_{l_6} &= \bigwedge \bb{C}\psi_{5}\\
	\call{H}_{l_7} &= \bigwedge \bb{C}(\psi_{6} \oplus \psi_7) & \call{H}_{l_8} &= \bigwedge (\bb{C} \psi_8 \oplus \bb{C} \psi_9)
	\end{align*}
These fit together to give the following
\begin{align*}
	\call{H}_\pi &= \call{H}_{l_1} \otimes \call{H}_{l_2} \otimes \call{H}_{l_3} \otimes \call{H}_{l_4} \otimes \call{H}_{l_5} \otimes \call{H}_{l_6} \otimes \call{H}_{l_7} \otimes \call{H}_{l_8}\\
	&\cong \bigotimes_{i = 1}^9\bigwedge \bb{C}\psi_i
	\end{align*}
Hence, the stabilisers of $\pi$ are given as follows.
\begin{align}
	\begin{split}
		S_\pi = \lbrace &X_{1}^{l_2}X^{l_3}_3, X_4^{l_3}X_2^{l_4}, X^{l_3}_3X^{l_3}_8, X^{l_3}_4X^{l_8}_9, X^{l_7}_6X^{l_3}_8, X^{l_7}_7X^{l_8}_9, X^{l_6}_{5}X^{l_7}_7, X^{l_6}_5X^{l_7}_{6}\rbrace
	\end{split}
	\end{align}
Our model of $\pi$ is $\llbracket \pi \rrbracket = (\call{H}_\pi, S_\pi)$.

In general, given a quantum error-correcting code $(\call{H}, S)$, if the codespace $\call{H}^S$ admits $n$ independent operators and the state space $\call{H}$ is $2^m$ dimensional, then the dimension of $\call{H}^S$ is $2^{m - n}$ \cite[Page 456, Proposition 10.5]{nc}. In the current situation, the total space has dimension $2^9$ and there are $8$ independent operators in $S$. Hence the codespace $\operatorname{Code}\pi$ is $2$ dimensional, indeed it is equal to the following span.
	\begin{equation}
		\operatorname{Code}\pi = \operatorname{Span}\big\lbrace \ket{+++++++++}, \ket{---------}\big\rbrace
	\end{equation}
\end{example}

\begin{example}\label{ex:cording_errors}
	Consider the following proof net $\pi$ whose red edges form an $m$-redex which when reduced yields the consequential proof net which we label $\pi'$.

\begin{equation*}\adjustbox{scale=0.90}{\begin{tikzcd}[column sep = small, row sep = small]
	& {\stackrel{l_1}{\ax}} &&&& {\stackrel{l_2}{\ax}} &&& {\stackrel{l_3}{\ax}} \\
	{\neg W_1} && W_2 && {\neg W_9} && W_{10} & {\neg W_5} && W_6 \\
	{\operatorname{c}} &&& {\stackrel{l_4}{\otimes}} &&& {\operatorname{c}} && {\stackrel{l_5}{\parr}} \\
	&&& {W_3 \otimes \neg W_8} &&&&& {\neg W_4 \parr W_7} \\
	&&&&&& {\stackrel{l_6}{\cut}}
	\arrow[curve={height=12pt}, no head, from=1-2, to=2-1]
	\arrow[draw={rgb,255:red,253;green,53;blue,73}, curve={height=-12pt}, no head, from=1-2, to=2-3]
	\arrow[from=2-1, to=3-1]
	\arrow[draw={rgb,255:red,253;green,53;blue,73}, curve={height=12pt}, no head, from=1-6, to=2-5]
	\arrow[curve={height=-12pt}, no head, from=1-6, to=2-7]
	\arrow[from=2-7, to=3-7]
	\arrow[draw={rgb,255:red,253;green,53;blue,73}, curve={height=-12pt}, from=2-5, to=3-4]
	\arrow[draw={rgb,255:red,253;green,53;blue,73}, curve={height=12pt}, from=2-3, to=3-4]
	\arrow[draw={rgb,255:red,253;green,53;blue,73}, curve={height=12pt}, no head, from=1-9, to=2-8]
	\arrow[draw={rgb,255:red,253;green,53;blue,73}, curve={height=-12pt}, no head, from=1-9, to=2-10]
	\arrow[draw={rgb,255:red,253;green,53;blue,73}, curve={height=-12pt}, from=2-10, to=3-9]
	\arrow[draw={rgb,255:red,253;green,53;blue,73}, curve={height=12pt}, from=2-8, to=3-9]
	\arrow[draw={rgb,255:red,253;green,53;blue,73}, no head, from=3-9, to=4-9]
	\arrow[draw={rgb,255:red,253;green,53;blue,73}, curve={height=-12pt}, from=4-9, to=5-7]
	\arrow[draw={rgb,255:red,253;green,53;blue,73}, curve={height=12pt}, from=4-4, to=5-7]
	\arrow[draw={rgb,255:red,253;green,53;blue,73}, no head, from=3-4, to=4-4]
\end{tikzcd}}
\end{equation*}
\begin{equation*}
	% https://q.uiver.app/?q=WzAsMTMsWzEsMCwiXFxzdGFja3JlbHtsXzF9e1xcYXh9Il0sWzAsMSwiXFxuZWcgWCJdLFsyLDEsIlgiXSxbMCwyLCJcXG9wZXJhdG9ybmFtZXtjfSJdLFs4LDEsIlxcbmVnIFgiXSxbMTAsMSwiWCJdLFs5LDAsIlxcc3RhY2tyZWx7bF8zfXtcXGF4fSJdLFsxMCwyLCJcXG9wZXJhdG9ybmFtZXtjfSJdLFs0LDEsIlxcbmVnIFgiXSxbNiwxLCJYIl0sWzUsMCwiXFxzdGFja3JlbHtsXzJ9e1xcYXh9Il0sWzMsMiwiXFxzdGFja3JlbHtsXzEnfXtcXGN1dH0iXSxbNywyLCJcXHN0YWNrcmVse2xfMid9e1xcY3V0fSJdLFswLDEsIiIsMCx7ImN1cnZlIjoyLCJzdHlsZSI6eyJoZWFkIjp7Im5hbWUiOiJub25lIn19fV0sWzAsMiwiIiwyLHsiY3VydmUiOi0yLCJzdHlsZSI6eyJoZWFkIjp7Im5hbWUiOiJub25lIn19fV0sWzEsM10sWzYsNCwiIiwwLHsiY3VydmUiOjIsInN0eWxlIjp7ImhlYWQiOnsibmFtZSI6Im5vbmUifX19XSxbNiw1LCIiLDIseyJjdXJ2ZSI6LTIsInN0eWxlIjp7ImhlYWQiOnsibmFtZSI6Im5vbmUifX19XSxbNSw3XSxbMTAsOCwiIiwyLHsiY3VydmUiOjIsInN0eWxlIjp7ImhlYWQiOnsibmFtZSI6Im5vbmUifX19XSxbMTAsOSwiIiwwLHsiY3VydmUiOi0yLCJzdHlsZSI6eyJoZWFkIjp7Im5hbWUiOiJub25lIn19fV0sWzIsMTEsIiIsMix7ImN1cnZlIjoyfV0sWzgsMTEsIiIsMSx7ImN1cnZlIjotMn1dLFs0LDEyLCIiLDAseyJjdXJ2ZSI6LTJ9XSxbOSwxMiwiIiwxLHsiY3VydmUiOjJ9XV0=
\begin{tikzcd}[column sep = small, row sep = small]
		& {\stackrel{l_1}{\ax}} &&&& {\stackrel{l_2}{\ax}} &&&& {\stackrel{l_3}{\ax}} \\
		{\neg W_1} && W_2 && {\neg W_3} && W_4 && {\neg W_5} && W_6 \\
		{\operatorname{c}} &&& {\stackrel{l_1'}{\cut}} &&&& {\stackrel{l_2'}{\cut}} &&& {\operatorname{c}}
		\arrow[curve={height=12pt}, no head, from=1-2, to=2-1]
		\arrow[curve={height=-12pt}, no head, from=1-2, to=2-3]
		\arrow[from=2-1, to=3-1]
		\arrow[curve={height=12pt}, no head, from=1-10, to=2-9]
		\arrow[curve={height=-12pt}, no head, from=1-10, to=2-11]
		\arrow[from=2-11, to=3-11]
		\arrow[curve={height=12pt}, no head, from=1-6, to=2-5]
		\arrow[curve={height=-12pt}, no head, from=1-6, to=2-7]
		\arrow[curve={height=12pt}, from=2-3, to=3-4]
		\arrow[curve={height=-12pt}, from=2-5, to=3-4]
		\arrow[curve={height=-12pt}, from=2-9, to=3-8]
		\arrow[curve={height=12pt}, from=2-7, to=3-8]
	\end{tikzcd}
	\end{equation*}
Both $\pi$ and $\pi'$ have a unique linear qubit ordering $i < j \Rightarrow W_i < W_j$, for both proof nets. We begin by considering the maps
\begin{align*}
	\bigwedge \bb{C} \psi^{1'} &\lto \bigwedge \bb{C} \psi^5_1 \otimes \bigwedge \bb{C} \psi^{6}_{1} \otimes \bigwedge \bb{C} \psi^{4}_1\\
	\bigwedge \bb{C} \psi^{2'} &\lto \bigwedge \bb{C} \psi^4_1 \otimes \bigwedge \bb{C} \psi_1^6 \otimes \bigwedge \bb{C} \psi_1^5
	\end{align*}
both defined by linearity along with the rule
\begin{equation}
	\ket{j} \longmapsto \frac{1}{2}(\ket{+++} + (-1)^j\ket{---})
	\end{equation}
Taking the tensor of the identity functions for the remaining links and composing with the canonical isomorphisms we obtain a morphism $\gamma: \call{H}_{\pi'} \lto \call{H}_{\pi}$.

The set $C_{\pi}$ is
\begin{equation}
	C_{\pi} = \{ X_1^{4}X_1^6, X_2^4X_2^6, X_1^5X_1^6, X_2^5X_2^6 \}
	\end{equation}
and we see that the image of $\gamma$ is contained within the subspace $\call{H}_{\pi}^{C_\pi}$ as predicted by Lemma \ref{lem:isomorphism}.

Furthermore, the operator $X^1X^4_1$ is an element of $S_{\pi} \setminus C_{\pi}$. As predicted by theorem \ref{thm:cut_model}, there exists a unique operator $X^1X^{1'}$ in $S_{\pi'}$ which makes \eqref{eq:cut_model_square} commute.
\end{example}

\appendix

\section{Majorana chains}
Majorana fermions and ordinary (Dirac) fermions are distinguished by whether their creation and anihilation operators $\gamma^\dagger, \gamma$ are self-adjoint $\gamma = \gamma^\dagger$ (Majorana) or not $\gamma \neq \gamma^\dagger$ (Dirac). There is a standard way to create Majorana quasi-particles from ordinary fermions, and ordinary fermionic quasi-particles from Majorana fermions. In the notation of \cite{LB} the Majorana chain is constructed from $L$ Dirac fermion creation and annihilation operators $c_j, c_j^\dagger$ ($1 \leq j \leq L$).
% https://q.uiver.app/?q=WzAsMTYsWzAsMSwiXFxidWxsZXQiXSxbMCwwLCIxIl0sWzAsMiwiY18xLCBjXzFeXFxkYWdnZXIiXSxbMSwxLCJcXGJ1bGxldCJdLFsxLDAsIjIiXSxbMSwyLCJjXzIsIGNfMl5cXGRhZ2dlciJdLFszLDEsIlxcYnVsbGV0Il0sWzMsMCwiaiJdLFsyLDAsIlxcaGRvdHMiXSxbMiwyLCJcXGhkb3RzIl0sWzMsMiwiY19qLCBjX2peXFxkYWdnZXIiXSxbNCwyLCJcXGhkb3RzIl0sWzQsMCwiXFxoZG90cyJdLFs1LDAsIkwiXSxbNSwxLCJcXGJ1bGxldCJdLFs1LDIsImNfTCwgY19MXlxcZGFnZ2VyIl0sWzAsMywiIiwwLHsic3R5bGUiOnsiaGVhZCI6eyJuYW1lIjoibm9uZSJ9fX1dLFszLDYsIiIsMCx7InN0eWxlIjp7ImhlYWQiOnsibmFtZSI6Im5vbmUifX19XSxbNiwxNCwiIiwwLHsic3R5bGUiOnsiaGVhZCI6eyJuYW1lIjoibm9uZSJ9fX1dXQ==
\[\begin{tikzcd}
	1 & 2 & \hdots & j & \hdots & L \\
	\bullet & \bullet && \bullet && \bullet \\
	{c_1, c_1^\dagger} & {c_2, c_2^\dagger} & \hdots & {c_j, c_j^\dagger} & \hdots & {c_L, c_L^\dagger}
	\arrow[no head, from=2-1, to=2-2]
	\arrow[no head, from=2-2, to=2-4]
	\arrow[no head, from=2-4, to=2-6]
\end{tikzcd}\]
and has Hamiltonian
\begin{equation}
	H_{\operatorname{MC}} = \sum_{i = 1}^{L-1}(c_i^\dagger c_{i+1} + c_{i+1}^\dagger c_i + c_ic_{i+1} + c_{i+1}^\dagger c_i^\dagger)
	\end{equation}
This Hamiltonian can be expressed more simply by defining:
\begin{equation}
	\gamma_{2j}^{\operatorname{LB}} = i(c_j^\dagger - c_j)\quad \gamma_{2j - 1}^{\operatorname{LB}}\quad 1 \leq j \leq L \quad S_j = -i\gamma_{2j}^{\operatorname{LB}}\gamma_{2j+1}^{\operatorname{LB}}
	\end{equation}
so that
\begin{equation}
	H_{\operatorname{MC}} = -\sum_{j = 1}^{L-1}S_j
	\end{equation}
Notice that $\gamma_{2j}^{\operatorname{LB}}, \gamma_{2j-1}^{\operatorname{LB}}$ are Majorana fermions $(\gamma_{2j}^{\operatorname{LB}})^\dagger = \gamma_{2j-1}^{\operatorname{LB}}$.
			
\begin{lemma}
	The operator $S_j$ is the stabiliser $X_jX_{j+1}$ of the quantum wire.
	\end{lemma}
\begin{proof}
	For simplicity, we consider the case of a pair of qubits $\bb{C}^2 \otimes \bb{C}^2$. Then
	\begin{equation}
		\ket{10} = c_j^\dagger \ket{00},\quad \ket{01} = c_{j+1}^\dagger \ket{00},\quad \ket{11} = c_j^\dagger c_{j+1}^\dagger \ket{00}
		\end{equation}
	So with $a,b \in \{ 0,1 \}$ we have
	\begin{equation}
		\ket{ab} = (c_j^\dagger)^a(c_{j+1}^\dagger)^b \ket{00}
		\end{equation}
	Thus:
	\begin{align*}
		\gamma_{2j+1}^{\operatorname{LB}}\ket{ab} &= (c_{j+1}^\dagger + c_{j+1})(c_j^\dagger)^a(c_{j+1}^\dagger)^b \ket{00}\\
		&= (-1)^a(c_j^\dagger)^a(c_{j+1} + c_{j+1})(c_{j+1})^b \ket{00}\\
		&= (-1)^a\ket{a\overline{b}}
		\end{align*}
	where $\overline{0} = 1, \overline{1}=0$.
	
	Similarly, $\gamma_{2j}^{\operatorname{LB}}\ket{ab} = i(c_j^\dagger - c_j)\ket{ab} = i(-1)^a \ket{\overline{a}b}$. Hence:
	\begin{align*}
		S_j\ket{ab} &= -i\gamma_{2j}^{\operatorname{LB}}\gamma_{2j+1}^{\operatorname{LB}}\ket{ab}\\
		&= -i \gamma_{2j}^{\operatorname{LB}}(-1)^a\ket{a\overline{b}}\\
		&= \ket{\overline{a}\overline{b}}
		\end{align*}
	\end{proof}

\bibliographystyle{amsalpha}
\providecommand{\bysame}{\leavevmode\hbox to3em{\hrulefill}\thinspace}
\providecommand{\href}[2]{#2}

\end{document}